\documentclass[reqno,11pt]{amsart}
\usepackage{latexsym}
\usepackage{amsthm,amsmath,amssymb}
\usepackage[english]{babel}
\usepackage{enumerate}
\usepackage{graphicx}
\usepackage{color}
\usepackage{a4wide}
\usepackage{hyperref}

\newcommand{\real}{\mathbb{R}}

\newcommand{\p}[1]{\left(#1\right)}

\newcommand{\curlyp}[1]{\left\{ #1 \right\}}
\newcommand{\abs}[1]{\left|#1\right|}
\newcommand{\norm}[1]{\left\|#1\right\|}
\newcommand{\commentout}[1]{}

\newcommand{\der}[2]{ \frac{\partial #1}{\partial #2} }
\newcommand{\prob}{\mathcal{P}}

\newcommand{\D}{\mathcal{D}}
\newcommand{\C}{\mathcal{C}}
\newcommand{\M}{\mathcal{M}}
\newcommand{\PP}{\mathcal P}
\newcommand{\R}{\mathbb{R}}
\newcommand{\N}{\mathbb{N}}
\newcommand{\dist}{\text{\rm dist}}

\newtheorem{lemma}{Lemma}

\newtheorem{definition}{Definition}
\newtheorem{proposition}{Proposition}
\newtheorem{theorem}{Theorem}
\newtheorem{remark}{Remark}

\begin{document}
\title{Nonlocal interactions by repulsive-attractive potentials: radial ins/stability}

\author{D. Balagu\'e $^1$, J. A. Carrillo$^2$,  T. Laurent$^3$ and G. Raoul$^4$ }

\address{$^1$ Departament de
Matem\`a\-ti\-ques, Universitat Aut\`onoma de Barcelona, E-08193
Bellaterra, Spain. E-mail: {\tt dbalague@mat.uab.cat}.}

\address{$^2$ ICREA and Departament de
Matem\`a\-ti\-ques, Universitat Aut\`onoma de Barcelona, E-08193
Bellaterra, Spain. E-mail: {\tt carrillo@mat.uab.es}. {\it On
leave from:} Department of Mathematics, Imperial College London,
London SW7 2AZ, UK.}

\address{$^3$ Department of Mathematics, University of California -
Riverside, Riverside, CA 92521,  USA. E-mail: {\tt
laurent@math.ucr.edu}.}

\address{$^4$DAMTP, University of Cambridge, Wilberforce road, CB3 0WA, United Kingdom. E-mail: {\tt
g.raoul@damtp.cam.ac.uk}.}

\maketitle

\begin{abstract}
We investigate nonlocal interaction equations with
repulsive-attractive radial potentials. Such equations describe
the evolution of a continuum density of particles in which they
repulse each other in the short range and attract each other in
the long range. We prove that under some conditions on the
potential, radially symmetric solutions converge exponentially
fast in some transport distance toward a spherical shell
stationary state. Otherwise we prove that it is not possible for a
radially symmetric solution to converge weakly toward the
spherical shell stationary state. We also investigate under which
condition it is possible for a non-radially symmetric solution to
converge toward a singular stationary state supported on a general
hypersurface. Finally we provide a detailed analysis of the
specific case of the repulsive-attractive power law potential as
well as numerical results.
\end{abstract}

\section{Introduction}
\label{sec:1}

Nonlocal interaction equations are continuum models for large
systems of particles where every single particle can interact not
only with its immediate neighbors but also with particles far
away. These equations have a wide range of applications. In
biology they are used to model the collective behavior of a large
number of individuals, such as a swarm of insects, a flock of
birds, a school of fish or a colony of bacteria  \cite{ME99,TB,
TBL, CuckerSmale1, CuckerSmale2,Tadmor2008, Dorsogna3,CFRT,CCR,
BT1,BT2,Bjorn,Alethea,DP,BDP,BCM,BCC}. In these models individuals
sense each other at a distance, either directly by sound, sight or
smell, or indirectly via chemicals, vibrations, or other signals.
Nonlocal interaction equations also arise in various contexts in
physics. They are used in models describing the evolution of
vortex densities in superconductors
\cite{Weinan:1994,SandierSerfaty, SandierSerfatybook,LinZhang,
AS,AMS,Mainini,DZ, Masmoudi}. They also appear in the modeling of
dynamics of agglomerating particles in two dimensions  (with loose
links to the one-dimensional sticky particles system)
\cite{Poupaud1}.  They also appear in simplified inelastic
interaction models for granular media
\cite{BenedettoCagliotiPulvirenti97,CMRV2,tosc:gran:00,LT}. Going
back to  biology,  nonlocal interaction equations arise also in
the modeling of the  orientational distribution of F-actin
filaments in  cells \cite{F-actin, Stevens1,Stevens2}.

In their simplest form, nonlocal interaction equations can be
written as
\begin{gather}
\der{\mu}{t} + \text{div}(\mu v) =0  \quad , \quad  v= - \nabla W*
\mu   \label{pdes1}
\end{gather}
where $\mu(t,x)=\mu_t(x)$ is the probability or mass density of
particles at time $t$ and at location $x \in \real^N$, $W: \real^N
\to \real$ is the interaction potential and $v(t,x)$ is the
velocity of the particles. We will always assume that the
interaction potential $W(x)=k(|x|)$ is radial and $C^2$- or
$C^3$-smooth away from the origin, depending on the results.
Typically the potentials we will consider have a singularity at
the origin.

When the potential $W$ is purely attractive, i.e. $W$ is a
radially symmetric increasing  function, then the density of
particles collapse on itself and converge to a Dirac Delta
function located at the center  of mass of the density. This Dirac
Delta function is the unique stable steady state and it is a
global attractor \cite{CDFLS}. The collapse toward the Dirac Delta
function can take place in finite time if the interaction
potential is singular enough at the origin and several works have
been recently devoted to the understanding of these singular
measure solutions \cite{BL,BCL,CDFLS,BGL}.

In biological applications however, it is often the case that
individuals attract each other in the long range in order to
remain in a cohesive group, but repulse each other in the short
range in order to avoid collision \cite{Mogilner2003,biobook}.
This lead to the choice of a radially symmetric potential $W$
which is first decreasing then increasing as a function of the
radius. We refer to these type of potentials as
repulsive-attractive potentials. Compared with the purely
attracting case where solutions always converge to a single Delta
function, nonlocal interaction equations with repulsive-attractive
potentials lead to solutions converging to possibly complex steady
states.  As such,  nonlocal interaction equations with
repulsive-attractive potentials can be considered  as a minimal
model for pattern formation in large groups of individuals.

Whereas nonlocal interaction equations with purely attractive
potential have been intensively studied there are still relatively
few rigorous results about nonlocal interaction equations with
repulsive-attractive potential. The 1D case has been studied in a
series of works \cite{FellnerRaoul1,FellnerRaoul2,Raoul}. The
authors have shown that the behavior of the solution depends
highly on the regularity of the interaction potential: for regular
interaction, the solution converges to a sum of Dirac masses,
whereas for singular repulsive potential, the solution remains
uniformly bounded. They also showed that combining a singular
repulsive with a smooth attractive potential leads to integrable
stationary states. Pattern formation in multi-dimensions have
recently been studied in \cite{KSUB,BUKB}. In these two works, the
authors perform a numerical study of the finite particle version
of \eqref{pdes1} and show that a repulsive-attractive potential
can lead to the emergence of surprisingly complex patterns. To
study these patterns they plug in \eqref{pdes1} an ansatz  which
is a distribution supported on a surface. This give rise to an
evolution equation for the surface. They then perform a linear
stability analysis around the uniform distribution on the sphere
and derive simple conditions on the potential which classify the
different instabilities. The various instability modes dictate
toward which pattern the solution will converge. They also check
numerically that what is true for the surface evolution equation
also holds for the continuum model \eqref{pdes1}. In another
recent work \cite{FHK} the specific case where the repulsive part
of the potential is the Newtonian potential and the attractive
part is polynomial is analyzed showing the existence of radially
compactly supported integrable stationary states. They also study
their nonlinear stability for particular cases.

In this paper we focus primarily on proving rigorous results about
the convergence of radially symmetric solutions toward spherical
shell stationary states in multi-dimensions.
\begin{definition}[Spherical Shell]\label{delta}
The  spherical shell of radius $R$, denoted $\delta_R$, is the
probability measure which is uniformly distributed  on the sphere
$\partial B(0,R) =\{x\in \real^N: \abs{x}=R\}$.
\end{definition}
Given a repulsive-attractive radial potential whose attractive
force does not decay too fast at infinity,  there always exists an
$R>0$ so that the spherical shell of radius $R$ is a stationary
state as it will be remarked below. One need then to address the
question of wether or not this spherical shell is stable. It is
classical, see \cite{AGS,CMRV2,MR1964483,CMRV1}, that the equation
\eqref{pdes1} is a gradient flow of the interaction energy
$$
 E[\mu]= \frac{1}{2} \iint_{\real^N \times \real^N} W(x-y) d \mu(x) d \mu(y)
$$
with respect to the euclidean Wasserstein distance. Thus, stable
steady states of \eqref{pdes1} are expected to be local minimizers
of the interaction energy. Simple energetic arguments will show
that in order for the spherical shell of radius $R$ to be a local
minimum of the interaction energy,  it is necessary that the
potential W satisfies:
\begin{enumerate}

\item[{\bf (C0)}] Repulsive-Attractive Balance: $\omega(R,R)=0$,

\item[{\bf (C1)}] Fattening Stability: $\partial_1 \omega(R,R)\le
0$,

\item[{\bf (C2)}] Shifting Stability: $\partial_1
\omega(R,R)+\partial_2 \omega(R,R)\le 0$,
\end{enumerate}
where the function $\omega:\real^2_+\longrightarrow \real$ is
defined by
\begin{equation} \label{omega-def2}
\omega(r,\eta)= -\frac{1}{ \sigma_{N}} \int_{\partial B(0,1)}
\nabla W (r e_1 - \eta y) \cdot e_1 \,d \sigma(y),
\end{equation}
$\sigma_N$ is the area of the unit ball in $\real^N$, $e_1$ is the
first vector of the canonical basis of $\real^N$, $d\sigma$
denotes the volume element of the manifold where the integral is
performed and $\real^2_+=(0,+\infty)\times(0,+\infty)$. Condition
{\bf (C0)} simply guarantees that the spherical shell $\delta_R$
is a critical point of the interaction energy. We will see that if
condition {\bf (C1)} is not satisfied then it is energetically
favorable to split the spherical shell into two spherical shells.
Heuristically this indicate that the density of particles, rather
than remaining on the sphere, is going to expand and occupy a
domain in $\real^N$ of positive Lebesgue measure. If condition
{\bf (C1)} is not satisfied we will therefore say that the
``fattening instability" holds. It can be easily checked that if
$\omega(R,R)=0$, then $\partial_1 \omega(R,R)$ is simply the value
of the divergence of the velocity field on the sphere of radius
$R$. So the fattening instability corresponds to an expanding
velocity field on the support of the steady state. We will also
see that if condition {\bf (C2)} is not satisfied it is
energetically favorable to increase or decrease the radius of the
spherical shell. This instability will be referred as the ``shift
instability''.

We now outline the structure of the paper and describe the main
results. In the preliminary section, section \ref{sec:2}, we
derive {\bf (C0)}--{\bf (C2)} from an energetic point of view and
we show that they correspond to avoiding the fattening and shift
instability. We also study the regularity of the kernel $\omega$
defined by \eqref{omega-def2}. A good understanding of the
regularity of $\omega$ will be necessary for later sections. We
also remind the reader of previous results from
\cite{BLR,BalagueCarrillo} about well posedness of \eqref{pdes1}
in $L^p(\real^N)$. Section \ref{sec:3} is devoted to a detailed
study of the fattening instability, both in the radially symmetric
case and in the non-radially symmetric case. We first show that if
condition {\bf (C1)} is not satisfied then it is not possible for
a radially symmetric $L^p$-solution to converge weakly-$*$ as
measures toward a spherical shell stationary state. We then
investigate singular stationary states supported on hypersurfaces
which are not necessarily  spheres. Such steady states have been
observed in numerical simulations \cite{KSUB,BUKB}. We show that
if the divergence of the velocity field generated by such
stationary state is positive everywhere on their support, then it
is not possible for an $L^p$-solution to converge toward the
stationary state in the sense of the topology defined by
$d_\infty$. Here $d_\infty$ stands for the infinity-Wasserstein
distance on the space of probability measures (see section
\ref{sec:3} for a definition). We also show that  if the
repulsive-attractive potential  $W$ is  singular  enough at the
origin, for example $W(x)\sim -|x|^b/b$ as $|x| \to 0$ with $b \le
3-N$, then the potential is so repulsive in the short range that
solutions can not concentrate on an hypersurface,  and this is
independent of  how attractive is the potential in the long range.
To be more precise we show that for potentials with such a strong
repulsive singularity at the origin, $L^p$ solutions can not
converge with respect to the $d_\infty$-topology  toward singular
steady states supported on hypersurfaces.

   Whereas section \ref{sec:3} is devoted to instability
results, section \ref{sec:4} is devoted to stability results. We
show that if {\bf (C0)}--{\bf (C2)} hold with strict inequalities,
then a radially symmetric solution of \eqref{pdes1} which starts
close enough  to the spherical shell  in the $d_\infty$ topology
will converge exponentially fast toward it. Under additional
assumptions on the potential we can also prove convergence with
respect to the $d_\alpha$ topology, $\alpha \in [1,+\infty)$. In
order for the stability results of section \ref{sec:4} to hold a
certain amount of regularity on the solutions is necessary.
Unfortunately weak $L^p$-solutions do not have this amount of
regularity. This is why in  section \ref{sec:5} we prove well
posedness of classical $C^1$-solutions. This covers a gap in the
existing literature which mostly considers weak solutions. The
results of  section \ref{sec:4} are true for this class of
classical $C^1$-solutions. The aim of section \ref{sec:6} is to
show  examples of how to apply the general instability and
stability theory in the case of power-law repulsive-attractive
potentials:
\begin{equation} \label{poweri}
W(x)=\frac{\abs{x}^a }{a }-\frac{\abs{x}^b }{b } \qquad 2-N<b <a.
\end{equation}
For this family of potentials, conditions {\bf (C0)}--{\bf (C2)}
can be explicitly formulated in terms of $a$ and $b$, therefore
leading to an explicit bifurcation diagram for the stability  of
the spherical shell in $\real^N$. Finally in the last section,
section \ref{sec:7}, we perform numerical computations of radially
symmetric solutions of \eqref{pdes1} with power-law potential
\eqref{poweri} and study their convergence toward spherical shell
stationary state. Since a spherical shell is a highly singular
function, it is challenging to perform such computations with
traditional methods. This is why, following
\cite{GosseToscani,BCC,MR2566595,FellnerRaoul1,FellnerRaoul2},
rather than simulating \eqref{pdes1} directly, we simulate the
evolution of the inverse of the cumulative distribution of the
radial measure associated to $\mu$. Since the inverse of the
cumulative distribution of a spherical shell is a constant
function, this approach has the virtue of smoothing the dynamics
and this provides us with a robust numerical scheme. Our numerical
simulations indicate the possible existence of integrable radial
stationary states stable under radial perturbations in the
parameter area corresponding to the fattening instability for
power-law repulsive-attractive potentials, an issue that will be
analysed elsewhere. This has already been proved in the particular
case of $b=2-N$ and $a\geq 2$ in \cite{FHK}.


\section{Preliminary section}
\label{sec:2}
\subsection{Radially symmetric formulation of the equation}

\begin{definition}[Radial Measures]\label{definition:radial}
We denote by $\prob^{r}(\real^N)$  the space of radially symmetric
probability measures. If  $\mu \in \prob^{r}(\real^N)$ then
$\hat{\mu} \in \prob([0,+\infty))$ is defined by
$$
 \int_{r_1}^{r_2} d\hat{\mu}(r)  = \int_{r_1< \abs{x}<r_2} d\mu(x)
 \qquad \mbox{and} \qquad
 \int_{0}^{r_2} d\hat{\mu}(r)  = \int_{0\leq \abs{x}<r_2} d\mu(x)
$$
for all $0<r_1<r_2$. We endow this space with the standard
weak-$*$ topology.
\end{definition}
Recall that $\delta_R \in \prob^r(\real^N)$ stands for the
spherical shell of radius $R$ (see Definition \ref{delta}). The
velocity field at point $x$ generated by a spherical shell of
radius $R$ is given by $v_R(x)=- \nabla W * \delta_R(x)$. Since
$W$ is radially symmetric, then by symmetry there exists a
function $\omega(r,\eta)$ such that
\begin{equation} \label{omega-def}
v_R(x)=- \nabla W * \delta_R(x)= \omega( \abs{x},R)
\frac{x}{\abs{x}}
\end{equation}
and one can easily check that this function $\omega$ is defined by
\eqref{omega-def2}, see \cite{BCL} for more details. Note also
that if $\mu \in \prob^{r}(\real^N)$ then it can be written as a
sum of spherical shells, $\mu=\int_0^{+\infty} \delta_{\eta} \;
d\hat{\mu}(\eta)$, and we conclude that
$$
-(\nabla W * \mu)(x)= -\int_0^{+\infty} (\nabla W  *
\delta_{\eta})(x) d\hat{\mu}(\eta)=  \int_0^{+\infty}
\omega(\abs{x},\eta) d\hat{\mu}(\eta) \, \frac{x}{\abs{x}}.
$$
Given $T>0$, $C([0,T];\prob^{r}(\real^N))$ denotes the set of
continuous curves of radial measures where continuity is with
respect to the weak-$*$ convergence. We say that $\mu \in
C([0,T];\prob^{r}(\real^N))$ is a radially symmetric solution of
\eqref{pdes1} if $\hat{\mu} \in C([0,T];\prob([0,+\infty))$
satisfies the one dimensional conservation law:
\begin{align}
&\partial_t \hat{\mu} + \partial_r (\hat{\mu} \hat{v})=0 \label{pde1r}\\
&\hat{v}(t,r)= \int_0^{+\infty} \omega(r,\eta) d\hat{\mu}_t(\eta)
\label{pde2r} \,,
\end{align}
in the distributional sense. We will now give conditions for the
velocity field to be well-defined by studying the properties of
the function $\omega$.

\subsection{Regularity of the function $\omega(r,\eta)$}

Let us remind that we assume that $W$ is radially symmetric and
belongs to $C^2(\real^N \backslash \{0\})$. The function
$\omega(r,\eta)$ defined by \eqref{omega-def2} is clearly $C^1$
away from the diagonal ${\mathcal D}=\{(r,r): r>0 \}$. Moreover,
the derivatives of $\omega$ are given by
\begin{equation} \label{omega-derr}
\partial_1\omega(r,\eta)= -\frac{1}{ \sigma_{N}}
\int_{\partial B(0,1)} \frac{\partial^2 W}{\partial x_1^2} (r e_1
- \eta y) \,d \sigma(y),
\end{equation}
and
\begin{equation} \label{omega-dereta}
\partial_2\omega(r,\eta)= \frac{1}{
\sigma_{N}} \int_{\partial B(0,1)} \nabla \left(\frac{\partial
W}{\partial x_1}\right) (r e_1 - \eta y) \cdot y \,d \sigma(y),
\end{equation}
away from the diagonal. We need to investigate the behavior of
$\omega$ on the diagonal.  Let us make the following definition:
\begin{definition}[Integrability on hypersurfaces]
A  radially symmetric function $g \in \C( \real^N \backslash
\{0\})$ is said to be locally integrable on hypersurfaces if
\begin{equation*}
\int_{[0,1]^{N-1}} \abs{g(\hat{x},0)} \; d\hat{x} < + \infty
\end{equation*}
where $\hat x =(x_1, \dots, x_{N-1})$, or equivalently, if
$\hat{g}(r)r^{N-2}$ is integrable on $(0,1)$ with
$g(x)=\hat{g}(|x|)$. By an abuse of notation, we sometimes say $\hat{g}(r)$ is integrable on hypersurfaces.
\end{definition}

\begin{lemma}[Regularity of the function $\omega$]\label{regsphere}
Let $W(x)=k(|x|)$ be  a radially symmetric potential  belonging to
$C^3(\real^N \backslash \{0\})$.
\begin{enumerate}
\item[(i)] If $k'(r)$ is  locally integrable on
hypersurfaces then $\omega \in C(\real^2_+)$.

\item[(ii)] If $k'(r)$, $k''(r)$, and $r^{-1}k'(r)$ are locally
integrable on hypersurfaces then $\omega \in C^1(\real^2_+)$.

\item[(iii)] Suppose $\Delta W$ is negative in a neighborhood of
the origin.  If $k'(r)$ is  locally integrable on hypersurfaces
but $\Delta W = k''+(N-1)r^{-1} k'$ is not, then for any $R>0$,
\begin{equation} \label{meme}
\lim_{\substack{(r,\eta) \notin \D\\{(r,\eta) \to
(R,R)}}}\partial_1 \omega (r,\eta)=+\infty.
\end{equation}

\end{enumerate}
\end{lemma}

Before proving the above lemma, let us discuss  the result.
Obviously the regularity of the function $\omega$ depends only on
the behavior of $W$ at the origin. Assume for simplicity that in
the neighborhood of the origin, the potential $W$ is a powerlaw,
that is $W(x)=k(|x|)=-|x|^b/b$ for all $x \in B(0,\varepsilon)$,
where $b$ is possibly negative. Note that  $k'(r)<0$ for $r<
\varepsilon$ so the potential is repulsive in the short range.
Lemma \ref{regsphere} then claims that $\omega$ is continuous if
$b>2-N$ and continuously differentiable if $b>3-N$. Statement
(iii) says that if $2-N<b \le 3-N$,  then $\omega$ is continuous
but its first derivative goes to $+\infty$ as $(r,\eta)$
approaches the diagonal.

Finally, let us remark that(i) is sharp in the sense that the
Newtonian potential $\abs{x}^{2-N}$ is the critical one for the
integrability on hypersurfaces. Precisely, Newton's Theorem
asserts that the function $\omega$ associated to the Newtonian
potential is discontinuous, it has a singularity, across the
spherical shell. We now prove the Lemma:
\begin{proof}
Let us prove (i). The function $\omega(r,\eta)$ can be rewritten
as
$$
\omega(r,\eta)= \int_{\partial B(0,\eta)} e_1 \cdot \nabla W (r
e_1 - y) \, \frac1{\sigma_N \eta^{N-1}} \, d\sigma(y).
$$
Seeing $\omega$ as a function of $x=r e_1$ and $\eta$, we can
apply Lemma \ref{reg2} from the appendix with $\mathcal
M_{\eta}:=\partial B(0,{\eta})$, $\phi_{\eta}{(x)}:=(\sigma_N
\eta^{N-1})^{-1}$, and $G(x):= e_1 \cdot \nabla W (x)$. Since
$|G(x)|$ is bounded by  $|k'(|x|)|$ which is locally integrable on
hypersurfaces, we  obtain that $\omega \in C(\real^2_+)$.

We now turn to the proof of (ii). It is simple to check that
$$
\frac{\partial^2 W}{\partial x_i \partial x_j} = k''(r) \frac{x_i
x_j}{r^2}+k'(r)\frac{\delta_{ij}}{r} - k'(r) \frac{x_ix_j}{r^3}
$$
and then $|\frac{\partial^2 W}{\partial x_i \partial x_j}|$ is
bounded by a radial function which is locally integrable on hypersurfaces
given by a linear combination of $k''(r)$ and $r^{-1}k'(r)$.
Moreover, it has the regularity needed in Lemma \ref{reg2}. We now
rewrite the derivatives $\partial_1 \omega(r,\eta)$ and
$\partial_2 \omega(r,\eta)$ in \eqref{omega-derr} and
\eqref{omega-dereta} as
\begin{align*}
&\partial_1\omega(r,\eta)= -\frac{1}{ \sigma_{N} \eta^{N-1}}
\int_{\partial B(0,\eta)} \frac{\partial^2 W}{\partial x_1^2} (r
e_1 - y) \,d \sigma(y),
\\
&\partial_2\omega(r,\eta)= \frac{1}{
\sigma_{N} \eta^N} \int_{\partial B(0,\eta)} \nabla
\left(\frac{\partial W}{\partial x_1}\right) (r e_1 - y) \cdot y
\,d \sigma(y) \,.
\end{align*}
The reader can easily check that Lemma  \ref{reg2}
applies similarly as before, so that $\omega\in C^1(\real^2_+)$.

Finally we prove (iii). Taking the divergence of \eqref{omega-def}
we obtain:
\begin{equation}\label{divdelta}
(\text{div } v_R) (x)  =- \Delta W * \delta_{R}(x)=\partial_1
\omega( \abs{x},R) + (N-1)\frac{ \omega(\abs{x},R)}{|x|},
\end{equation}
and therefore $\partial_1 \omega (r,\eta)$ can be written:
$$
\partial_1 \omega (r,\eta)=- (\Delta W * \delta_\eta) (r e_1) - (N-1) \frac{\omega(r,\eta)}{r}.
$$
For $0<\varepsilon<r_0$, let $\chi_\varepsilon\in C^\infty(\mathbb
R_+)$ be a cut-off function, such that $\chi_\varepsilon=1$ on
$[0,\varepsilon/2]$, and $\chi_\varepsilon=0$ on
$[\varepsilon,\infty)$. Choose $\varepsilon$ such that  the
function $-\Delta W^\varepsilon(x):= - \chi_{\varepsilon}(x)
\Delta W(x)$ is nonnegative. Using  Lemma \ref{notreg} with
$\eta_1$ and $\eta_2$ such that $\eta_1 < R < \eta_2$,  and noting
that $\dist(r e_1,\mathcal \partial B(0,\eta))=|r-\eta|$, we find
that
\begin{equation*}
\lim_{\substack{(r,\eta) \notin \D\\{(r,\eta) \to
(R,R)}}}- (\Delta W^\varepsilon * \delta_\eta) (r e_1)= +\infty.
\end{equation*}
To conclude the proof, note that the functions  $(r,\eta) \mapsto
\frac{\omega(r,\eta)}{r}$ and  $(r,\eta) \mapsto
([(1-\chi_{\varepsilon}) \Delta  W ]* \delta_\eta) (r e_1)$ are
bounded in a neighborhood of $(R,R)$.
\end{proof}


\subsection{The two radial instabilities}
In this subsection we exhibit some elementary calculations in
order to understand under which conditions a spherical shell is a
stable steady state. Rigorous results about stability and
instability of spherical shell with respect to the transport
distance will be provided in section 3 and 4. This subsection
provide motivations for the rigorous results to come later.

\begin{definition}[Steady states]
A probability measure $\mu \in \prob(\real^N)$ is said to be a
steady state of the nonlocal interaction equation \eqref{pdes1} if
$$
-(\nabla W * \mu)(x) = 0 \qquad \text{ for all } x \in \text{\rm supp}(\mu).
$$
\end{definition}

We now show that if the attractive strength $k'(r)$ of a
repulsive-attractive potential $W(x)=k(|x|)$ does not decay faster
$1/r^{N}$ as $r \to \infty$, then there  exists a spherical shell
steady state.

\begin{lemma}[Existence of spherical shell steady states]
Let $W(x)=k(|x|)$ be a radially symmetric potential belonging to
$C^1(\real^N \backslash \{0\})$ and such that $k'(r)$ is locally
integrable on hypersurfaces. Let us assume that the potential is
repulsive-attractive in the following sense: there exists $R_a>0$
such that
$$
k'(r)\ge 0 \text{ for } r>R_a, \qquad \mbox{and} \qquad k'(r)<0
\text{ for }0<r< R_a.
$$
Defining for $r > 2R_a$ the function
\begin{equation} \label{sigma}
\Sigma(r) := \inf_{r/2\leq s \leq 2r} k'(s) \ge 0,
\end{equation}
we will further assume that
$$
\lim_{r\to \infty} r^{N} \Sigma(r) = +\infty .
$$
Then there exists at least a $R>0$ such that the spherical shell
$\delta_R \in \prob(\real^N)$ is a steady state to \eqref{pdes1}.
\end{lemma}

\begin{proof}
Note that from \eqref{omega-def} we directly obtain that a
spherical shell $\delta_R \in \prob(\real^N)$ is a steady state if
and only if $\omega(R,R)=0$, that is, if and only if condition
{\bf (C0)} holds. Since $k'(r)$ is locally integrable on
hypersurface $\omega \in C(\real^2_+)$ due to Lemma
\ref{regsphere}.  So the function $F(r):=\omega(r,r)\in
C(\real_+)$. Using formula \eqref{omega-def2}, we get
$$
F(r)=\frac{1}{ \sigma_{N}} \int_{\partial B(0,1)} k'(r|y-e_1|)
\frac{y-e_1}{|y-e_1|} \cdot e_1 \,d \sigma(y) \, .
$$
Let us remark that $(y-e_1)\cdot e_1 \leq 0$ for all $y\in\partial
B(0,1)$, and thus for $2r< R_a$ we easily get $F(r)>0$. It is
enough to show that there exists $r> 2R_a$ such that $F(r)<0$. In
order to do this, we proceed as in \cite[Proposition 2.2]{CDFLS2}
and divide the integral in the definition of $F(r)$ into two sets:
$A:=\partial B(0,1)\cap B(e_1, R_a/r)$ and its complementary set
$A^c$. Note that the integrand is positive on $A$ and negative on
$A^c$. We will show that for $r$ large enough the integral over
the set $A^c$ is greater in absolute value than the integral over
the set $A$. It is easy to see that $B:=\{y \in \partial B(0,1)
\mbox{ such that } 2|y-e_1|\geq 1\}\subset A^c$ and
$B\neq\emptyset$ as soon as $r> 2R_a$.

We first estimate the integral
$$
\left|\int_{A} k'(r|y-e_1|)
\frac{y-e_1}{|y-e_1|} \cdot e_1 \,d \sigma(y)\right|=\int_{A} |k'(r|y-e_1|)|
\frac{(e_1-y)\cdot e_1}{|e_1-y|}  \,d \sigma(y)
$$
Let $\theta$ be the angle between $e_1-y$ and $e_1$ and note that
for all $y\in A:=\partial B(0,1)\cap B(e_1, R_a/r)$ we have by the
law of cosines
$$
\frac{e_1-y}{|e_1-y|} \cdot e_1 = \cos \theta \leq
\frac{R_a}{2r}.
$$
Using Lemma \ref{propmanifolds} from the Appendix with
$\M=\partial B(0,1)$  we then obtain
\begin{align*}
\left|\int_{A} k'(r|y-e_1|)
\frac{e_1-y}{|e_1-y|} \cdot e_1 \,d \sigma(y)\right|
&\le \frac{R_a}{2r} \int_{A} |k'(r|y -e_1|)| d \sigma(y)\\
&\le \frac{R_a}{2r} \int_0^{R_a/r} |k'(rs)|  \; |  \M \cap \{|y-e_1|=s\} |_{{\mathcal H}^{N-2}}  \; ds\\
&\le C \frac{R_a}{2r} \int_0^{R_a/r} |k'(rs)|  \;s^{N-2}  ds\\
&=  C \frac{R_a}{2r^N} \int_0^{R_a} |k'(z)|z^{N-2} dz
\le  \frac{C_1}{r^{N}}
\end{align*}
where we have used the fact that $k'(r)$ is integrable on
hypersurfaces to obtain the last inequality.

Since the integrand is negative in $A^c$ and since $B \subset A^c$
for $r>2R_a$ we have:
\begin{equation*}
\int_{A^c}k'(r|y-e_1|)\frac{y-e_1}{|y-e_1|}\cdot e_1\,d \sigma(y)\leq\int_{B}k'(r|y-e_1|)\frac{y-e_1}{|y-e_1|}\cdot e_1\,d \sigma(y).
\end{equation*}
Moreover, thanks to the law of cosines,
$\frac{y-e_1}{|y-e_1|}\cdot e_1=\cos(-\theta)\leq -1/4$ for $y\in
B$, and then using \eqref{sigma}
\begin{align*}
 \int_{B}k'(r|y-e_1|)\frac{y-e_1}{|y-e_1|}\cdot e_1\,d \sigma(y)
&\le  - \frac 14 \int_{B} k'(r|y-e_1|) \,d \sigma(y) \\
&\le - \frac 14 \int_{B} \,d \sigma(y) \, \Sigma(r) \; \le \;  -C_2 \;  \Sigma(r).
\end{align*}
Condition \eqref{sigma} on $\Sigma(r)$ implies that $C_2 \Sigma(r)
\ge C_1/r^{N}$ for $r$ large enough and therefore $F(r)<0$ for $r$
large enough. Then the continuity of $F$ implies the existence of
a radius $\tilde r>0$ such that $F(\tilde r)=0$.
\end{proof}

The following proposition gives some hints about the stability
properties of the spherical shell steady states.

\begin{proposition}[Instability modes by energy arguments] \label{energetic}
Assume that the radial interaction potential $W$ is such that
$\omega \in C^1(\real^2_+)$ and let $\delta_{R}$ be a steady
state, that is $\omega(R,R)=0$.
\begin{enumerate}
\item[(i)] If {\rm\bf (C1)} is not satisfied then by splitting the
spherical shell into two spherical shells we can decrease the
energy. More precisely   there exists $dr_0>0$ such that, given
$0<\abs{dr}< dr_0$,
$$
E[(1-\epsilon)\delta_{R}+ \epsilon \delta_{R+dr}]< E[\delta_{R}]
$$
if  $\epsilon$ is small enough.

\item[(ii)] If {\rm\bf (C2)} is not satisfied then by increasing
or decreasing the radius of the spherical shell we can decrease
the energy. More precisely there exists $dr_0>0$  such that
$$
E[\delta_{R+dr}]< E[\delta_{R}]
$$
for all $0<\abs{dr}< dr_0$.
\end{enumerate}
\end{proposition}
\begin{proof} Let us introduce the notations
$$
 E[\mu,\nu]:= \frac{1}{2} \iint_{\real^N \times \real^N} W(x-y) d \mu(x) d \nu(y),
$$
so that $E[\mu,\mu]=E[\mu]$, and
$$
 E(r,\eta):= E[ \delta_{r} , \delta_{\eta} ]
 =\frac{1}{2}\frac{1}{\sigma_{N}^2}
  \iint_{\partial B(0,1) \times \partial B(0,1)}  W(r x-\eta y)  \, d\sigma(x) d
  \sigma(y)\,.
$$
Taking the derivative we get:
\begin{align*}
\der{E}{r}(r,\eta)&=\frac{1}{2}\frac{1}{\sigma_{N}^2}
  \iint_{\partial B(0,1) \times \partial B(0,1)}  \nabla W(r x-\eta y) \cdot x \, d\sigma(x) d \sigma(y)\\
  &=\frac{1}{2}\frac{1}{\sigma_{N}}
   \int_{\partial B(0,1)}  \left(  \frac{1}{\sigma_{N}}   \int_{\partial B(0,1)}  \nabla W(r x-\eta y) d \sigma(y)  \right)   \cdot x \, d\sigma(x) \\
    &=\frac{1}{2}\frac{1}{\sigma_{N}}
   \int_{\partial B(0,1)}  \Big( \nabla W *\delta_{\eta} (r x)   \Big)   \cdot x \, d\sigma(x) =-
   \frac{1}{2}\omega(r,\eta)\,.
\end{align*}
Since $E(r,\eta)=E(\eta,r)$, the Hessian matrix of $E(r,\eta)$ is
given by
$$
H(r,\eta)=-\frac{1}{2}\left[ \begin{array}{cc}
\partial_1{\omega}(r,\eta) &\partial_2{\omega}(r,\eta)\\
\partial_2{\omega}(\eta,r) &\partial_1{\omega}(\eta,r)
\end{array}\right].
$$
If $\delta_{R}$ is a steady state, i.e. $\omega(R,R)=0$, then $\nabla E(R,R)=0$ and
\begin{gather}\label{papa}
E(R+dr,R) = E(R,R)- \frac{1}{4} \partial_1 \omega(R,R) dr^2+o(dr^2) \\
\label{maman}
E(R+dr,R+dr) = E(R,R) - \frac{1}{2}(\partial_1 \omega(R,R)+\partial_2 \omega (R,R)) \; dr^2  + o(dr^2)
\end{gather}
The proof of (ii) follows directly from the Taylor expansion
\eqref{maman}. By using the Taylor expansions \eqref{papa} and
\eqref{maman} as well as  the bilinearity of $E[\mu]=E[\mu,\mu]$:
\begin{align*}
E\Big[(1-\epsilon)\delta_{\partial B(0,R)}&+ \epsilon \delta_{\partial B(0,R+dr)}\; , \; (1-\epsilon)\delta_{\partial B(0,R)}+ \epsilon \delta_{\partial B(0,R+dr)}\Big]\\
=&(1-\epsilon)^2E(R,R)+2\epsilon(1-\epsilon)E(R+dr,R)+\epsilon^2 E(R+dr,R+dr)\\
=& E(R,R)- \frac{\epsilon}{2} \partial_1\omega(R,R) dr^2-
\frac{\epsilon^2}{2}  \partial_2 \omega (R,R)) \; dr^2  + o(dr^2)
\end{align*}
from which (i) follows by taking  $\epsilon$ and $dr_0$ small enough.
\end{proof}

The following elementary Lemma shows that the instability
condition $\partial_1 \omega(R,R)>0$ (i.e. {\rm\bf (C1)} is not
satisfied) simply means that the divergence of the velocity field
generated by the spherical shell is positive on the spherical
shell. Being the velocity field ``expanding", it makes sense that
splitting the spherical shell into two reduces the energy as
proven in previous Proposition \ref{energetic}.

\begin{lemma}[Divergence of the velocity field] \label{div-lemma}
Assume the spherical shell $\delta_{R}$ is a steady state, i.e.,
condition {\rm\bf (C0)}: $\omega(R,R)=0$. Let $v_R$ be the
velocity field generated by $\delta_R$, given by
\eqref{omega-def}. Then
\begin{equation} \label{div22}
(\text{\rm div } v_R) (x)= \partial_1 \omega(R,R) \qquad \text{for
all } x\in \partial B(0,R).
\end{equation}
\end{lemma}
\begin{proof}
This is a direct consequence of \eqref{divdelta} together with the fact
 that $\omega(R,R)=0$.
\end{proof}

\subsection{Well-posedness of $L^p$-solutions}
Global existence and uniqueness of  $L^p$-solutions of equation
\eqref{pdes1} was established in \cite[Theorem 1]{BLR} under some
conditions on the interaction potential $W$:

\begin{theorem}[$L^p$-Well posedness theory]\label{BLR}
Consider $1<q<\infty$ and $p$ its H\"older conjugate. Suppose
$\nabla W\in {\mathcal W}^{1,q}(\real^N)$ and  $\mu_0 \in
L^p(\real^N)\cap \prob_2(\real^N)$ is nonnegative. Then  there
exists a time $T^* > 0$ and a nonnegative function $\mu\in
C([0,T^*],L^p(\real^N))\cap C^1([0,T^*],{\mathcal W}^{-1,p}
(\real^N))$ such that \eqref{pdes1} holds in the sense of
distributions in $\real^N \times (0,T^*)$ with $\mu(0) = \mu_0$.
Moreover the second moment of $x \mapsto \mu(t,x)$ remains bounded
and the $L^1$ norm is conserved. Also, the function $t \to
\norm{\mu(t)}_{L^p}^p$ is differentiable and satisfies
\begin{equation}
\frac{d}{dt}\curlyp{\|\mu(t)\|_{L^{p}}^p} = - (p-1) \int_{\real^N}
\mu(t,x)^p \mbox{\rm div }  v(t,x) \; dx\label{lp-equality} \qquad
\forall t \in [0,T^*].
\end{equation}
Furthermore, if ${\rm ess~sup~} \Delta W < +\infty$, then  $t \to
\norm{\mu(t)}_{L^p}^p$ does not grow faster than exponentially and
we have global well-posedness.
\end{theorem}
In the above theorem $\prob_2(\real^N)$ stands for the space of
probability measure with finite second moment.   We will refer to
the solutions provided by the above theorem as {\it
$L^p$-solutions}.

One can find in \cite{BalagueCarrillo} that the authors extend the
global-in-time well posedness $L^p$-theory to repulsive-attractive
potentials under suitable conditions. We summarize the result in
the following theorem.

\begin{theorem}[Dealing with possibly growing at $\infty$ attractive potentials]
Assume that $W(x)=k(|x|)$ is a radially symmetric
repulsive-attractive potential,
$W(x)=(W_R+W_A)(x)=(k_R+k_A)(\abs{x})=k(\abs{x})$ with $k\in
C^2((0,+\infty))$, $W_A$ attractive (i.e. $k_A'>0$), with $\nabla
W\in {\mathcal W}_{loc}^{1,q}(\real^N)$, $1<q<\infty$, and $W_R$
compactly supported repulsive ($k_R'\leq 0$). Furthermore, assume
that $k$ satisfies:
\begin{enumerate}[(i)]
\item  $\exists \delta_1 >0$ such that $k''(r)$ is monotonic in $(0,\delta_1)$.
\item $\exists \delta_2>0$ such that $rk''(r)$ is monotonic in $(0,\delta_2)$.
\item $D:=\sup_{r\in(0,\infty)}\abs{k'_R(r)}<\infty$
\item There exists $m$ such that $\frac{k_A(r)}{1+r^m}$ is bounded and increasing.
\end{enumerate}
Then there exists a global in time solution for the equation
\eqref{pdes1} with compactly supported initial data $\mu_0\in
L^p(\real^N)$, which is compactly supported for all $t\geq 0$.
\end{theorem}

\section{The Fattening instability and dimensionality of the steady state}
\label{sec:3}

\subsection{The radially symmetric case}

This first subsection concerns radially symmetric solutions. We
show that if the singularity of $W$ at the origin is such that the
kernel $\omega$ is $C^1$, and if condition {\bf (C1)} is not
satisfied, then a radially symmetric solution can not converge
weakly toward the spherical shell stationary state. We also show
that the same result holds if the singularity of $W$ at the origin
is so strong that the kernel $\omega$ is not $C^1$ (and this is
independent of how strong the attractive part of the potential
is).

\begin{theorem}[Instability of spherical shells: radially symmetric case]   \label{fat-inst1}
Let $W(x)=k(|x|)$ be  a radially symmetric potential  belonging to
$C^3(\real^N \backslash \{0\})$ and such that $k'(r)$ is locally
integrable on hypersurface (so that $\omega$ is continuous).
Assume that the spherical shell $\delta_{R}$ is a  steady state,
that is, {\bf (C0)}: $\omega(R,R)=0$, and that one of the two
following hypotheses hold:
\begin{enumerate}
\item[{\rm (i)}]  $k''(r)$ and $r^{-1}k'(r)$ are locally
integrable on hypersurfaces (so that $\omega$ is $C^1$), and
 $\partial_1
\omega(R,R)>0$. \item[{\rm (ii)}] $\Delta W$ is negative in a
neighborhood of the origin and is not locally integrable on
hypersurfaces (in which case $\omega$ is not $C^1$ and
$\lim_{\substack{(r,\eta) \notin \D\\{(r,\eta) \to
(R,R)}}}\partial_1 \omega (r,\eta)=+\infty$).
\end{enumerate}
Then it is not possible for an $L^p$ radially symmetric solution
of \eqref{pde1r}-\eqref{pde2r} to converge weakly-$*$ as measures
to $\delta_{R}$ as $t\to \infty$.
\end{theorem}

To clarify the result, let us consider the case where the
repulsive-attractive potential $W(x)=k(|x|)$ has its repulsive
part described by  a powerlaw. Let say, for example, that
$k(r)=-r^b/b$,  for all $r <1$ and $k'(r)>0$ for $r>2$.   If $2-N
< b \le 3-N$, then $\Delta W$ is not locally integrable on
hypersurfaces and therefore, according to (ii), whatever is the
behavior of $W(x)$ for $|x|>1$, $L^p$ radially symmetric solutions
can not converge toward the steady state. In other words if the
repulsive singularity of the potential is equal to or stronger
than $\abs{x}^{3-N}$ then the potential is so repulsive in the
short range that solution can not concentrate on a spherical
shell, and this is independent of how attractive the potential is
in the long range. On the other hand if $b> 3-N$ then the kernel
$\omega$ is $C^{1}$. In this case, the balance between the
repulsive part and the attractive part of the potential dictates
whether or not the spherical shell is an attractor: if $\partial_1
\omega(R,R)>0$, then the repulsive part dominates and the
spherical shell is not an attractor.

We remind, see \cite{MR1964483,MR2059493,MR2219334}, that for
$1\le p < \infty$ the distance $d_p$ between two measures
$\nu,\,\rho$ is defined by
\[
    d_p^p(\nu,\rho)=\inf_{\pi\in\Pi(\nu,\rho)}\left\lbrace \int_{\R^N\times \R^N} |x-y|^pd\pi(x,y)\right\rbrace,
\]
where $\Pi(\nu,\rho)$ is the set of those joint distribution
functions with marginals $\nu$ and $\rho$. When $p=+\infty$ then
the distance is defined as
\[
    d_\infty(\nu,\rho) = \inf_{\mathcal T:\real^N\longrightarrow \real^N} \left\lbrace \sup_{y\in\R^N} \abs{y-\mathcal T (y)}\,:\, \mathcal T \#\rho = \nu\right\rbrace.
\]
We now prove the Theorem.
\begin{proof}
If conditions (ii) of the Theorem holds, then from \eqref{meme} of Lemma \ref{regsphere},  it is clear that there
exists $\delta>0$ such that
\begin{equation} \label{monotone}
\text{$\forall \eta \in (R-\delta,R+\delta)$, $r \to\omega(r,\eta)$ is strictly increasing in $(R-\delta,R+\delta)$}.
\end{equation}
Of course \eqref{monotone} also trivially holds if condition (i) of the Theorem is satisfied.
We proceed by contradiction. Assume that $\mu(x,t)=\mu_t(x)$ is an
$L^p$ radially symmetric solution which  converges weakly-$*$ as
measures to a spherical shell of radius $R$ as $t \to \infty$.

{\it Step 1.} Assume first that ${\mu}_t$ converges toward
$\delta_R$ not only weakly-$*$ as measures but also with respect
to the $d_\infty$-topology. This implies that the support of the
radial solution $\hat{\mu}_t$ to \eqref{pde1r} converge to the
point $\{R\}$. Choose $T>0$ such that $\text{supp}( \hat{\mu}_t)
\subset (R-\delta,R+\delta)$ for all $t>T$. Using the monotonicity
property \eqref{monotone} we obtain that for $ t\ge T$ and for
$R-\delta< r_1 < r_2<R+\delta$
$$
\hat{v}(t,r_2)-\hat{v}(t,r_1) =
 \int_{R-\delta}^{R+\delta} \omega(r_2,\eta)-  \omega(r_1,\eta) \; d\hat{ \mu}_t( \eta) \ge 0
$$
where $\hat{v}$ is the velocity field in radial coordinate defined
by \eqref{pde2r}. Therefore for all $t \ge T$ the function $r \to
\hat{v}(t,r)$ is increasing on $(R-\delta,R+\delta)$. Let $r_1(t)$
and $r_2(t)$ be two solutions of the ODE $r_i'(t)=
\hat{v}(t,r_i(t))$, $i=1,2$. Since
$$
\frac{d}{dt} (r_2(t)-r_1(t))^2= 2  (r_2(t)-r_1(t))
(\hat{v}(t,r_2(t))-\hat{v}(t,r_1(t))) \ge 0 \, ,
$$
we easily see that if  for some time $t\ge T$, $r_1(t)$ and
$r_2(t)$ are in $(R-\delta,R+\delta)$,  then their distance
increases. This contradicts the fact that the support of
$\hat{\mu}_t$ is converging to the point $\{R\}$ as $t \to
\infty$. Let us be more precise.  Since $\mu_T$ is supported in
$(R-\delta,R+\delta)$ and is absolutely continuous with respect to
the Lebesgue measure, there  exists $R_1$ and $R_2$ in
$(R-\delta,R+\delta)$, $R_1 \neq R_2$, such that $\int_{0}^{R_1}
\hat{\mu}_T(x) dx =1/3$ and $\int_{R_2}^\infty \hat{\mu}_T(x) dx
=1/3$.  Consider the ODEs $r_i'(t)= \hat{v}(t,r_i(t))$,
$r_i(T)=R_i$, $i=1,2$. Clearly $r_1(t)$ and $r_2(t)$ remain in
$(R-\delta,R+\delta)$ for all $t\ge T$ (otherwise the support of
$\mu_t$ would not stay in $(R-\delta,R+\delta)$). So
$\abs{r_2(t)-r_1(t)}\ge \abs{R_2-R_1}$ for all $t \ge T$ and  the
support of $\mu_t$ can not converge to the point $\{R\}$, which
contradicts our assumption.

{\it Step 2.} Assume now that ${\mu}_t$ converges weakly toward
$\delta_{R}$  but  does not converge  with respect to the
$d_\infty$-topology.  From the continuity of the function $\eta
\to \omega(r,\eta)$ together with \eqref{pde2r} it is clear that
$\hat{v}(r,t)$  converges pointwise to $\omega(r,R)$ as $t \to
\infty$.  Since the support of $\hat{\mu}_t$ does not converge to
the set $\{R\}$ there is a sequence of times at which there is
always non-zero  mass  in $(0,R-\epsilon) \cup
(R+\epsilon,+\infty)$. Since $\omega(R,R)=0$, the monotonicity
condition \eqref{monotone} implies that $\omega(r,R)<0$ for all $r
\in(R- \epsilon,R)$ and  $\omega(r,R)>0$ for all $r \in(R,R+
\epsilon)$ as long as $\epsilon<\delta$. Because of the pointwise
convergence of  $\hat{v}$ there exists a time $T>0$ such that  for
all $t> T$,  $\hat{v}(t,R-\epsilon)<0$ and
$\hat{v}(t,R+\epsilon)>0$. So after this time $T$ mass cannot
enter the region $[R-\epsilon, R+\epsilon]$. This together with
the existence of a time $t>T$ for which there is some mass in the
complementary of $[R-\epsilon, R+\epsilon]$ contradict the weak
convergence towards $\delta_R$.
\end{proof}

\begin{remark}
In Step {\rm 2} of this proof, since we are dealing with radially
symmetric solutions, the problem is essentially one dimensional
and the characteristics are ordered. This allows us to exclude the
possibility of an $L^p$ solution converging toward a spherical
shell even if this convergence is very weak and the support of the
solution does not converge. In the non radially symmetric case we
will be only able to exclude convergence in $d_\infty$.
\end{remark}

\subsection{The non-radially symmetric case}

In this subsection we consider non-radially symmetric solutions
and we investigate whether it is possible for an $L^p$-solution to
converge toward a steady state supported on an hypersurface which
not necessarily a sphere.  Indeed  in numerical simulations
\cite{KSUB,BUKB}, it is observed that depending on the choice of
the repulsive-attractive potential $W$, solutions of \eqref{pdes1}
can either converge to  steady states which are smooth densities
or to singular steady states which are measures supported on an
hypersurface. We consider steady states $\bar{\mu}$ of the form
 \begin{equation} \label{hyper}
\int_{\real^N} f(x) d\bar{\mu}(x) = \int_{\M} f(x) \phi(x) d \sigma(x) \quad \forall f \in C(\real^N)
\end{equation}
where $\M$ is a compact $C^2$ hypersurface and $d\sigma$ is the
volume element on $\M$. Roughly speaking, we prove that if the the
potential is as singular or more singular than $|x|^{3-N}$ at the
origin, then it is not possible for an $L^p$-solution to  converge
toward such a steady state with respect to the
$d_\infty$-topology.   We also prove that the same result holds if
the  potential is less singular than $|x|^{3-N}$, and if the
divergence of the velocity field generated by such steady state is
strictly positive on its support.

\begin{theorem}[Instability of Spherical Shells: Nonradial case] \label{Theorem:fat-inst2}
Let $W(x)=k(|x|)$ be a radially symmetric potential which belongs
to $C^2(\real^N \backslash \{0\})$.  Assume that $\lim_{r \to 0}
\widehat{\Delta W}(r)=-\infty$ and that close to the origin
$\widehat{\Delta W}(r)$ is monotone. Let $\bar{\mu}$ be a steady
state of the form \eqref{hyper} with $\M$ being a compact $C^2$
hypersurface and let $\bar v$ be the velocity field generated by
$\bar{\mu}$, that is $\bar v= - \nabla W * \bar{\mu}$. If one of
the two condition holds:
\begin{enumerate}
\item[(i)] $\Delta W$
is locally  integrable on hypersurfaces,  $\phi\in L^\infty(\M)$ and
\begin{equation} \label{div2}
(\text{\rm div } \bar v) (x):= - (\Delta W * \bar{\mu})(x) > 0
\qquad \text{for all } x\in \text{supp } \bar{\mu},
\end{equation}
\item[(ii)] $\Delta W$
is not locally integrable on hypersurfaces and $\phi(x) \ge \phi_0>0$ for all $x\in \M$,
\end{enumerate}
then it is not possible for an $L^p$ solution of \eqref{pdes1} to
converge to $\bar{\mu}$ with respect to the $d_\infty$-topology as
$t \to \infty$.
\end{theorem}

Before to prove this Theorem, let us make some remarks:

\begin{remark}
According to Lemma {\rm \ref{div-lemma}}, the result of Theorem
{\rm \ref{fat-inst1}} (i) of the previous subsection can be
reformulated as follows: assume that the spherical shell
$\delta_{R}$ is a steady state and let $v_R$ be its velocity
field. If
\begin{equation*} 
(\text{\rm div } v_R) (x)= -(\Delta W* \delta_{R})(x)>0 \qquad
\text{for all } x\in \partial B(0,R)
\end{equation*}
then it is not possible for an $L^p$ radially symmetric solution
to converge weakly-$*$ as measures to $\delta_R$ as $t\to \infty$.
So conditions (i) of Theorems {\rm \ref{fat-inst1}} and
{\rm\ref{Theorem:fat-inst2}} are essentially the same. Similarly
condition (ii) of both Theorems are also essentially the same.  In
this sense Theorem {\rm\ref{Theorem:fat-inst2}} can be seen as a
generalization of Theorem {\rm\ref{fat-inst1}} to the non-radially
symmetric case.
\end{remark}

\begin{remark}
The assumption $\lim_{r \to 0} \widehat{\Delta W}(r)=-\infty$
simply guarantees that the potential $W$ is strongly repulsive at
the origin. The monotonicity of  $\widehat{\Delta W}(r)$ in a
neighborhood of the origin is not  essential to the proof and
could be replaced by weaker hypotheses. But in practice all
potentials of interest satisfy this monotonicity condition.
\end{remark}

\begin{remark}
Part (ii) of the Theorem, roughly speaking, states that if the
repulsive-attractive potential  $W$ is more singular than
$|x|^{3-N}$ at the origin, then whatever is its attractive part,
it is not possible for an $L^p$ solution of \eqref{pdes1} to
converge with respect to the $d_\infty$-topology  toward a
singular steady state supported on an hypersurface. So we see that
the dimensionality of stable steady states depends on the degree
of singularity of the potential. For such potential with a strong
repulsive singularity at the origin, steady states are expected to
be absolutely continuous with respect to the Lebesgue measure.
\end{remark}

Theorem \ref{Theorem:fat-inst2} is  a direct consequence of the
three  Lemmas to follow.

\begin{lemma}[Approximating the divergence of the velocity field]
Let $W$ be as  stated in {\rm Theorem~\ref{Theorem:fat-inst2}} and
let $\bar{\mu}$ be a compactly supported probability measure not
belonging to $L^p$. Suppose there exists a H\"older continuous
function $\widetilde{\Delta W} \ge \Delta W$ such that
\begin{equation} \label{popo}
-\widetilde{\Delta W} *\bar{\mu} > 0 \qquad \text{on \rm
supp$(\bar{\mu})$}\,,
\end{equation}
then it is not possible for an $L^p$ solution of \eqref{pdes1} to
converge to $\bar{\mu}$ with respect to the $d_\infty$-topology as
$t \to \infty$.
\end{lemma}

\begin{proof}
We proceed by contradiction. Let $\mu_t$ be an $L^p$ solution such that
 $\lim_{t \to \infty}d_{\infty} (\mu_t, \bar{\mu})=0$.
We are going to show that there exists a $T>0$ and an $\epsilon>0$ such that for all $t>T$
 \begin{equation}\label{bob}
 (\Delta W * \mu_t)(x) < - \epsilon \qquad \text{for all $x\in \text{supp}( \mu_t)$},
 \end{equation}
 and combined with the equality \eqref{lp-equality}:
 $$
 \frac{d}{dt} \|\mu_t\|_{L^p}^p = (p-1) \int_{\real^N} (\Delta W * \mu_t) \mu_t^p dx,
 $$
this guarantees that a subsequence $\mu_{t_n}$ converges weakly in
$L^p$ to an $L^p$ function, which contradicts the assumption that
$\bar{\mu} \notin L^p$. Let us prove \eqref{bob}. Write
 \begin{equation}\label{kkkk}
 \Delta W * \mu_t=\widetilde{\Delta W} *(\mu_t-\bar{\mu}) + \widetilde{\Delta W} * \bar{\mu} + (\Delta W - \widetilde{\Delta W})* \mu_t
 \end{equation}
and note that since $\widetilde{\Delta W} \ge \Delta W$  the third
term is negative for all $t$ and $x$. Since $\widetilde{\Delta W}$
is continuous so is $\widetilde{\Delta W}* \bar{\mu}$ and,
therefore, \eqref{popo} implies that there exists an $\epsilon>0$
and an open set $\Omega$ containing the support of $\bar{\mu}$
such that $\widetilde{\Delta W}* \bar{\mu}< -\frac\epsilon 2$ on
$\Omega$. Here we used that the ${\rm supp}(\bar{\mu})$ is a
compact manifold. Note  that since  $\lim_{t \to \infty}d_{\infty}
(\mu_t, \bar{\mu})=0$ the support of $\mu_t$ will eventually be in
$\Omega$. To estimate the first term of \eqref{kkkk}, we consider
${\mathcal T}_t: \real^N \to \real^N$  a map pushing forward
$\mu_t$ to $\mu$, i.e. ${\mathcal T}_t \# \mu_t =\bar{\mu}$. Then
  \begin{align*}
   \norm{ \widetilde{\Delta W} *(\mu_t-\bar{\mu}) (x) }_{L^\infty(\real^N)} & \le \int_{\real^N} |\widetilde{\Delta W}(x-y)- \widetilde{\Delta W} (x- {\mathcal T}_t(y))| d \mu_t(y) \\
   & \le  \int_{\real^N}c  |y- {\mathcal T}_t(y))|^\beta d \mu_t(y),
  \end{align*}
since this inequality is true for any map ${\mathcal T}_t$ pushing
forward $\mu_t$ to $\mu$,
\begin{equation*}
 \| \widetilde{\Delta W} *(\mu_t-\bar{\mu}) \|_{L^\infty(\real^N)} \le c \;d_\infty(\mu_t,\bar{\mu})^\beta,
\end{equation*}
so that for $t\geq T$ with $T$ large enough,
$$
 \| \widetilde{\Delta W} *(\mu_t-\bar{\mu}) \|_{L^\infty(\real^N)} \le \frac \epsilon 4.
$$
Since $\widetilde{\Delta W} \ge \Delta W$, the last term of \eqref{kkkk} is negative, so that $\Delta W * \mu_t<0$.
\end{proof}

In order to conclude the proof we now need to show that under the
hypotheses of the theorem  there exists a H\"older continuous
function $\widetilde{\Delta W} \ge \Delta W$ satisfying
\eqref{popo}. Define
$$
\Delta W^\epsilon(x) := \begin{cases}
\Delta W (x) & \text{ if }|x| \ge \epsilon \\
\Delta W (\epsilon e_1) & \text{ if }|x| < \epsilon \\
\end{cases}.
$$
The function $\Delta W^\epsilon(x)$ is obviously H\"older
continuous and, due to the monotonicity of $\Delta W$ around the
origin we have $\Delta W^\epsilon\ge \Delta W$ for $\epsilon$
small enough. We are left to show that $\Delta W^\epsilon*
\bar{\mu}<0$ on the support of $\bar{\mu}$  for $\epsilon$ small
enough and this is done in the following two Lemmas.

\begin{lemma}[Continuity in $\epsilon$ of the divergence of the velocity field]
Let $\bar{\mu}$, $\M$ and $W$ be as  stated in {\rm
Theorem~\ref{Theorem:fat-inst2}} (i). Then $\Delta W^\epsilon
*\bar{\mu}$ converges uniformly on $\M$ toward $\Delta W
*\bar{\mu}$. Therefore  there exists $\epsilon>0$ such that
$\Delta W^\epsilon* \bar{\mu}<0$ on the support of $\bar{\mu}$.
\end{lemma}
\begin{proof}
Since $\M$ is $C^2$ and compact, using Lemma \ref{propmanifolds} from the appendix,
there exist constants $\delta,C_1,C_2>0$ so that
$$
C_1 \int_0^\epsilon g(r) r^{N-2} dr \le \int_{\M \cap B(x,\epsilon)} g(\abs{x-y}) d \sigma(y) \le C_2 \int_0^\epsilon g(r) r^{N-2} dr
$$
for all $x\in \M$, for all $\epsilon<\delta$ and for all nonnegative function
$g$ locally integrable on hypersurfaces. Since $ \Delta W(x)$ is radial and goes
to $-\infty$ monotonically  as $\abs{x} \to 0^+$, we clearly have
that $\Delta W(\epsilon e_1) - \Delta W(x-y) \ge 0$ for all $y \in
B(x,\epsilon)$ if $\epsilon$ is small enough.  Then we obtain
that, for all $x\in \M$,
\begin{align*}
\abs{(\Delta W^\epsilon *\bar{\mu})(x)-(\Delta W*\bar{\mu})(x)} &= \int_{\M \cap  B(x,\epsilon)} \Big[\Delta W(\epsilon e_1) - \Delta W(x-y) \Big] \phi(y) d\sigma(y)\\
& \le  C_2  \norm{\phi}_{L^\infty(\M)}  \int_{0}^{\epsilon} \Big[\widehat{\Delta W}(\epsilon) - \widehat{\Delta W}(r)\Big]  r^{N-2}dr \\
& \le  C_2  \norm{\phi}_{L^\infty(\M)}  \int_{0}^{\epsilon} \Big[-
\widehat{\Delta W}(r)\Big]  r^{N-2}dr
\end{align*}
and we conclude using the fact that $\Delta W$ is integrable on
hypersurfaces.
\end{proof}

\begin{lemma}
Let $\bar{\mu}$, $\M$ and $W$ be as  stated in {\rm Theorem~\ref{Theorem:fat-inst2}} (ii). Then there exists $\epsilon>0$ such that
$\Delta W^\epsilon*
\bar{\mu}<0$ on the support of $\bar{\mu}$.
\end{lemma}

\begin{proof}
Choose $r_0$ as in Lemma \ref{propmanifolds} and also small enough
so that $\widehat{\Delta W}(r) \le 0$ for all $r \le r_0$. For
$\epsilon < r_0$ we then have
\begin{equation}\label{deltaW}
(\Delta W^\epsilon*\bar\mu)(x) \le \int_{\epsilon \le |x-y| \le r_0} \Delta W (x-y) d\bar \mu(y)+ \int_{ |x-y| > r_0} \Delta W (x-y) d\bar \mu(y).
\end{equation}
Since $\Delta W$ is bounded on $B(0,2\,\textrm{diam}(\M))\setminus
B(0,r_0)$, the second term is uniformly bounded for $x\in \M$. We
use  Lemma \ref{propmanifolds} to estimate the first term of
\eqref{deltaW}:
\begin{multline*}
\int_{\epsilon \le |x-y| \le r_0}  \Delta W (x-y) d\bar \mu(y) = \int_{\mathcal M \cap \{y: \epsilon \le |x-y| \le r_0\}}   \widehat{\Delta W} (|x-y|) \phi(y) d\sigma(y) \\
 \leq \phi_0  \int_\epsilon^{r_0}  \widehat{\Delta W} (r) \abs{\mathcal M \cap \partial B(x,r) }_{{\mathcal H}^{N-2}} dr
   \leq  \tilde C \phi_0  \int_\epsilon^{r_0}  \widehat{\Delta W} (r)r^{N-2} dr,
\end{multline*}
and since $\Delta W$ is not locally integrable on hypersurface and
$ \widehat{\Delta W}<0$ on $[0,r_0]$, the last integral goes to
$-\infty$ as $\epsilon \to 0$. Then, for $\varepsilon>0$ small
enough, $\Delta W^\epsilon*\bar\mu<0$ on $\mbox{supp }(\bar\mu)$.
\end{proof}


\section{Stability for Radial Perturbations}
\label{sec:4} In this section, we give sufficient conditions for
the stability under radial perturbations of $\delta_R$ stationary
solutions in transport distances for the system
\eqref{pde1r}-\eqref{pde2r}. Let us denote by $\PP_2^r(\R^N)$ the
set of radial probability measures with bounded  second  moment.

Here, we will work with radial solutions with the following hypotheses of
minimal regularity {\bf (HMR)}: we assume that for any given
$\mu_0 \in \PP_2^r(\R^N)$, there exists $\mu\in
AC([0,T],\PP_2^r(\R^N))$, with $\mu_t=\mu_0$ for $t=0$, such that
$$
\hat{v}(t,r)= \int_0^{+\infty} \omega(r,\eta) d\hat{\mu}_t(\eta)
\in L^2((0,T)\times\R^N))
$$
for all $T>0$ and their corresponding radial measures
$\hat{\mu}_t$ satisfy \eqref{pde1r} in the weak distributional
sense. Moreover, they satisfy that $\int_0^\infty
r\,d\hat\mu_t(r)$ is an absolutely continuous function in time for
which
\begin{equation}
\frac{d}{dt} \int_0^\infty r\,d\hat\mu_t(r) = \int_0^\infty
\hat{v}(t,r)\,d\hat\mu_t(r) \label{newref1}
\end{equation}
holds a.e. $t\geq 0$. Furthermore, if $\hat{\mu}_0$ is compactly
supported, we assume that
$$
r_1(t)=\min\lbrace\textrm{supp }(\hat\mu_t)\rbrace \quad \mbox{and} \quad
r_2(t)=\max \lbrace\textrm{supp }(\hat\mu_t)\rbrace\, ,
$$
are absolutely continuous functions with $\frac d{dt}r_i(t) =
\hat{v}(t,r_i(t))$ a.e. $t\geq 0$, $i=1,2$.

The existence theory developed in Section~\ref{sec:5} ensures that
smooth classical solutions satisfying {\bf (HMR)} exist for $\mu_0
\in \PP_2^r(\R^N)\cap {\mathcal W}^{2,\infty}(\R^N)$ initial data
under suitable assumptions on the potential. Therefore, we assume
in this section that our radial solutions satisfy
\eqref{pde1r}-\eqref{pde2r} with $\omega$ given by
\eqref{omega-def2} verifying suitable hypotheses specified in each
result.

\begin{theorem}[Stability for local perturbations] \label{stability}
Assume $\omega\in C^1(\mathbb R_+^2)$ as given in
\eqref{omega-def2} and that $\delta_R$ is a stationary solution to
\eqref{pde1r}-\eqref{pde2r}, that is, the condition {\rm\bf (C0)}:
$\omega(R,R)=0$. Let us assume that {\rm\bf (C1)} and {\rm\bf
(C2)} are satisfied with strict inequality, that is:
\begin{equation}\label{stabvelocity}
\partial_1\omega(R,R)<0 \qquad \mbox{ and }\qquad
\partial_1\omega(R,R)+\partial_2\omega(R,R)<0\,.
\end{equation}
Then there exists $\varepsilon_0>0$ such that if the initial data
$\mu_0\in\PP_2^r(\R^N)$ satisfies ${\rm supp }(\hat\mu_0) \subset
[R-\varepsilon_0,R+\varepsilon_0]$, and for any solution to
\eqref{pde1r}-\eqref{pde2r} with initial data satisfying {\rm\bf
(HMR)} we get
\begin{equation*}
d_2(\hat\mu_t,\delta_R)\leq Ce^{-\gamma t},
\end{equation*}
for any $0<\gamma<-\max\left(\partial_1\omega(R,R),\,\frac
d{dR}\omega(R,R) \right)$ for suitable $C$.
\end{theorem}

\begin{proof}[Proof of the Theorem]
Since we have assumed that the solutions to
\eqref{pde1r}-\eqref{pde2r} satisfy the regularity conditions {\bf
(HMR)}, then $\Gamma(t):=\textrm{diam
}(\textrm{supp}(\hat\mu_t))=r_2(t)-r_1(t)$, and
$$
\Theta(t):=\int_0^\infty r\,d\hat\mu_t(r)-R \, ,
$$
are absolutely continuous function of $t\geq 0$. We will proceed
by contradiction.

We define $T:=\min\{t\geq 0;\, \Gamma(t)+|\Theta(t)|\geq
4\varepsilon_0\}$ and let us assume that $T<\infty$ for all
$\varepsilon_0>0$ close to 0. Note that $T>0$ by continuity of
$\Gamma(t)+|\Theta(t)|$, since $\textrm{supp }(\hat\mu_0) \subset
[R-\varepsilon_0,R+\varepsilon_0]$ implies that
\begin{equation*}
\Gamma(0)+|\Theta(0)|\leq
2\varepsilon_0+\int_{R-\varepsilon_0}^{R+\varepsilon_0}|r-R|
\,d\hat\mu_0(r)\leq 3\varepsilon_0 \, .
\end{equation*}
Now, for $t\in [0,T]$, $\textrm{supp}(\hat\mu_t)\subset
[R-4\varepsilon_0,R+4\varepsilon_0]$, since
\begin{align}
|r_i(t)-R| &\leq |r_i(t)-(R+\Theta(t))|+|\Theta(t)|\leq
(r_2(t)-r_1(t))+|\Theta(t)|\nonumber\\
&=\Gamma(t)+|\Theta(t)|\leq 4\varepsilon_0 \, ,\label{riR}
\end{align}
using that the center of mass $\Theta(t)+R$ is obviously in
$[r_1(t),r_2(t)]$, for all $t\geq 0$ and the definition of $T$.

Then, for $t\in [0,T]$, Taylor expanding to order one and using
that $\partial_1\omega$ is uniformly continuous on
$[R-4\varepsilon_0,R+4\varepsilon_0]^2$ together with {\bf (HMR)},
we get
\begin{align*}
\frac d{dt}\Gamma(t)&=\frac d{dt}r_2(t)-\frac d{dt}r_1(t) = \hat{v}(t,r_2(t))-\hat{v}(t,r_1(t))=\int_0^\infty \left[\omega(r_2(t),\eta)-\omega(r_1(t),\eta)\right]\,d\hat\mu_t(\eta)\\
&=\int_0^\infty
\left[\partial_1\omega(r_1(t),\eta)(r_2(t)-r_1(t))+g(r_1(t),r_2(t),\eta)\right]\,d\hat\mu_t(\eta),
\end{align*}
where $g$ satisfies
\begin{equation}\label{o}
\lim_{|r_2- r_1|\to 0}\left(\sup_{\eta\in [r_1,r_2]}\frac
{|g(r_1,r_2,\eta)|}{|r_2-r_1|}\right)=0.
\end{equation}
Since \eqref{o} is satisfied, the integral of $g$ can be estimated
as follows
\begin{equation*}
\int_0^\infty
g(r_1(t),r_2(t),\eta)\,d\hat\mu_t(\eta)=\int_{r_1(t)}^{r_2(t)}
g(r_1(t),r_2(t),\eta)\,d\hat\mu_t(\eta)=o(r_2(t)-r_1(t))=o(\Gamma(t)).
\end{equation*}
Proceeding with the same argument as before using \eqref{o}, we
can estimate
\begin{align*}
\frac d{dt}\Gamma(t)&=(r_2(t)-r_1(t))\int_0^\infty\partial_1\omega(r_1(t),\eta)\,d\hat\mu_t(\eta)+o(\Gamma)\\
&=(r_2(t)-r_1(t))\int_0^\infty\left[\partial_1\omega(R,R)+\big(\partial_1\omega(r_1(t),\eta)-\partial_1\omega(R,R)\big)\right]\,d\hat\mu_t(\eta)+o(\Gamma).
\end{align*}
Since $\eta\in \textrm{supp}(\hat\mu_t)=[r_1(t),r_2(t)]\subset
[R-4\varepsilon_0,R+4\varepsilon_0]$ thanks to \eqref{riR}, we can
then use the uniform continuity of $\partial_1 \omega$ on
$[R-4\varepsilon_0,R+4\varepsilon_0]^2$ to get:
\begin{equation*}
\abs{\partial_1\omega(r_1(t),\eta)-\partial_1\omega(R,R)}\le C{|r_1(t)-R|+|\eta-R|}\le C{|r_1(t)-R|+|r_2(t)-R|},
\end{equation*}
for any $\eta\in \textrm{supp}(\hat\mu_t)$. We can then use
\eqref{riR} again giving
\begin{align*}
\frac
d{dt}\Gamma(t)&=\partial_1\omega(R,R)(r_2(t)-r_1(t))+o(\Gamma)+o(|\Theta|).
\end{align*}

On the other hand, we can also estimate using \eqref{newref1}
\begin{align*}
\frac d{dt}\Theta(t)=&\,\int_0^\infty \hat{v}(r,t)\,d\hat\mu_t(r)=\int_0^\infty\!\!\int_0^\infty \omega(r,\eta)\,d\hat\mu_t(r)\,d\hat\mu_t(\eta)\\
=&\,\int_0^\infty\!\!\int_0^\infty
\left[\omega(\eta,\eta)+\partial_1\omega(\eta,\eta)(r-\eta)\right]
\,d\hat\mu_t(r)\,d\hat\mu_t(\eta)+o(\Gamma)+o(|\Theta|),
\end{align*}
where we have again used an argument as in \eqref{o} to estimate
the rest term of the Taylor expansion, and we use it once again to
obtain
\begin{align*}
\frac d{dt}\Theta(t)=&\,\int_0^\infty\!\!\int_0^\infty
\left[\omega(R,R)+\frac d{dR}\omega(R,R)(\eta-R)
+\partial_1\omega(\eta,\eta)(r-\eta)\right]\,d\hat\mu_t(r)\,d\hat\mu_t(\eta)\\
&+o(\Gamma)+o(|\Theta|)\\
=&\,\frac
d{dR}\omega(R,R)\left(\int_0^\infty\eta\,d\hat\mu_t(\eta)-R\right)+\partial_1\omega(R,R)\left(\int_0^\infty
r\,d\hat\mu_t(r)
-\int_0^\infty\eta\,d\hat\mu_t(\eta)\right)\\
&+o(\Gamma)+o(|\Theta|)\\
=&\,\left(\frac
d{dR}\omega(R,R)\right)\Theta+o(\Gamma)+o(|\Theta|).
\end{align*}

We now combine the estimates on $\Gamma$ and $\Theta$ to get:
\begin{equation}\label{gammatheta}
\frac
d{dt}\left(\Gamma+|\Theta|\right)(t)\leq\max\left(\partial_1\omega(R,R),\,\frac
d{dR}\omega(R,R)
\right)\left(\Gamma+|\Theta|\right)(t)+o\left(\Gamma+|\Theta|\right).
\end{equation}
Let us point out that all the
$o\left(\Gamma+|\Theta|\right)$-terms can be made uniformly small
in the interval $[0,T]$ by taking $\varepsilon_0$ small by their
definitions and using that $\textrm{supp}(\hat\mu_t)\subset
[R-4\varepsilon_0,R+4\varepsilon_0]$ in $[0,T]$. More precisely,
let $\gamma \in (0, -\max\left(\partial_1\omega(R,R),\,\frac
d{dR}\omega(R,R) \right))$. We can choose $\varepsilon_0>0$ small
enough for the rest terms of \eqref{gammatheta} to satisfy:
\begin{equation}\label{assrest}
\frac{o\left(\Gamma(t)+|\Theta(t)|\right)}{\Gamma(t)+|\Theta(t)|}\leq
\left|\max\left(\partial_1\omega(R,R),\,\frac d{dR}\omega(R,R)
\right)\right|-\gamma,
\end{equation}
for any $\Gamma(t),\Theta(t)$ since $\Gamma(t)+|\Theta(t)|\leq
4\varepsilon_0$ for all $t\in[0,T]$ due to \eqref{riR}. Then
\eqref{assrest} is satisfied for all $t\in [0,T]$, and thus,
\begin{equation*}
\frac d{dt}\left(\Gamma+|\Theta|\right)(t)\leq -\gamma
\left(\Gamma+|\Theta|\right)(t),
\end{equation*}
so that for $t\in [0,T]$,
\begin{equation}\label{gammatheta2}
\left(\Gamma+|\Theta|\right)(t)\leq\left(\Gamma+|\Theta|\right)(0)
e^{-\gamma\,t}\, .
\end{equation}
In particular, for any time $t\in [0,T]$,
$\left(\Gamma+|\Theta|\right)(t)\leq\left(\Gamma+|\Theta|\right)(0)\leq
3\varepsilon_0$ and thus, using the continuity of
$\left(\Gamma+|\Theta|\right)(t)$ since $T<+\infty$ we can
continue up to $\tilde{T}>T$ satisfying
$\left(\Gamma+|\Theta|\right)(t)\leq 4\varepsilon_0$ contradicting
the definition of $T$. Thus, $T=\infty$ for small enough
$\varepsilon_0$ and \eqref{gammatheta2} then holds for all $t\geq
0$. Thanks to \eqref{riR}, this implies the exponential
convergence of $d_2(\hat\mu_t,\delta_R)$ to $0$:
\begin{equation*}
d_2(\hat\mu_t,\delta_R)^2= \int_{0}^{\infty}(r-R)^2d\hat{\mu}_t(r)
\leq \max\left(|r_1(t)-R|^2,|r_2(t)-R|^2\right)\leq
\left(\Gamma+|\Theta|\right)^2(t)\leq 3\varepsilon_0
e^{-\gamma\,t}
\end{equation*}
for all $t\geq 0$.
\end{proof}

\begin{remark}
Lemma {\rm \ref{regsphere}} gives sufficient conditions to get the
assume regularity $\omega \in C^1(\real^2_+)$. Previous Theorem
holds for all radially symmetric potentials $W(x)=k(|x|)$
belonging to $C^3(\real^N \backslash \{0\})$ such that $k''(r)$
and $r^{-1}k'(r)$ are integrable on hypersurfaces. This applies
also to the next result for non local perturbations.
\end{remark}

\begin{remark}
The first part of condition \eqref{stabvelocity} implies intuitively that the
velocity field created by $\delta_R$ given by
$\omega(r,R)$ is decreasing at $r=R$ and therefore, particles are
pushed locally in space and in time towards radius $R$
for small perturbations.
\end{remark}

From now on, we denote by $\varphi(t,\cdot)$ the pseudo-inverse of
the distribution function of the radial measure $\hat \mu_t$, that
is
\begin{equation}\label{pseudoinverse}
\varphi(t,\xi)=\inf\left\{r\in\mathbb R_+;\,
\int_{0}^rd\hat\mu_t\geq \xi\right\}.
\end{equation}
$\varphi$ then satisfies
\begin{equation}\label{pi2}
\partial_t \varphi(t,\xi)=\hat{v}(t,\varphi(t,\xi))=\int_0^\infty \omega(\varphi(t,\xi),\eta)\,d\hat\mu_t(\eta).
\end{equation}
Note that by the definition of $\varphi$,
\begin{equation}\label{pi3}
\int_{[r_1,r_2]}\,d\hat\mu_t(\eta)=\int_{\{\xi;\,r_1\leq
\varphi(t,\xi)\leq r_2\}}\,d\xi.
\end{equation}
In the next theorem, we will work with solutions to system
\eqref{pde1r}-\eqref{pde2r} satisfying {\bf (HMR)} for which the
pseudo-inverse of the distribution function is an absolutely
continuous function on time satisfying \eqref{pi2} in the
classical sense a.e. in $t$. Solutions obtained in Section
\ref{sec:5} do satisfy these conditions.

\begin{theorem}[Stability: Tail control] \label{stability2}
Assume $\omega\in C^1(\mathbb R^2)$ and that $\delta_R$ is a
locally-stable stationary solution to \eqref{pde1r}-\eqref{pde2r},
that is, $\omega(R,R)=0$ and the local stability condition
\eqref{stabvelocity} holds. Assume moreover that the velocity
field associated to $\delta_R$ verifies
$$
\omega(r,R)>0 \mbox{ on } (0,R),\qquad \omega(r,R)<0 \mbox{ on }
(R,\infty), \qquad \mbox{and} \qquad \partial_{1}\omega(0,R)>0 \,,
$$
and the following long-range controls on the interaction
potential: for some $\alpha\ge 1$, there exists $\lambda>0$ such that
\begin{equation}\label{as1}
\omega(r,\eta)\leq \frac 1 \lambda-\lambda r^\alpha\textrm{ for
}(r,\eta)\in\mathbb R_+\times [R-\lambda,R+\lambda],
\end{equation}
\begin{equation}\label{as2}
\sup_{[0,\lambda]}\left|\partial_1\omega(\cdot,\eta)\right|\leq
\frac 1 \lambda (1+\eta^\alpha),
\end{equation}
\begin{equation}\label{as3}
|\omega(r,\eta)|\leq \frac 1
\lambda(1+r^\alpha)(1+\eta^\alpha)\textrm{ for }(r,\eta)\in
\mathbb R_+^2.
\end{equation}
Then, for any solution to \eqref{pde1r}-\eqref{pde2r} satisfying
{\rm\bf (HMR)} and \eqref{pi2} with initial data $\mu_0\in
\PP_2^r(\mathbb R^N)$ such that $\hat\mu_0(\{0\})=0$, and
$d_\alpha(\hat\mu_0,\delta_R)$ is small enough,
\begin{equation*}
\lim_{t\to\infty} d_\alpha(\hat\mu_t,\delta_R) = 0.
\end{equation*}
\end{theorem}

\begin{remark}
If we assume that the initial condition is compactly supported,
then the long-range controls \eqref{as1}, \eqref{as2}, \eqref{as3}
on the interaction potential are not
required anymore. Those are only necessary to control the behavior
of the tail of the distribution and its interaction with the rest.
\end{remark}

\begin{proof}[Proof of the Theorem]

\

\emph{Step 1.- ``Claim: Given
$\hat{\mu}\in\prob_{2}^{r}(\real^N)$. If
$d_\alpha(\hat\mu,\delta_R)$ is small, then the associated
velocity fields to $\hat\mu$ and $\delta_R$ share some confining
properties'':} For any $\vartheta>0$ small enough, thanks to our
assumptions on $\omega$, we can show that there exists $\Lambda>0$
such that if $d_\alpha(\hat\mu,\delta_R)\leq\Lambda$, then
\begin{equation}\label{setineq2}
\left\{\begin{array}{l}
\hat{v}(r)>C_1 \,r>0\textrm{ on }(0,\vartheta],\\[2mm]
\hat{v}(r)>v_1>0 \textrm{ on }[\vartheta,R-\vartheta],\\[2mm]
\hat{v}(r)<-v_1 \textrm{ on }[R+\vartheta,\infty),
\end{array}\right.
\end{equation}
where $\hat{v}(r)$ is the velocity field associated to $\hat\mu$
by \eqref{pde2r}. To prove the first inequality, notice that
$d_\alpha(\hat\mu,\delta_R)\leq\Lambda$ implies that
\begin{equation}\label{defepsilon}
\int_{[R-\sqrt\Lambda,R+\sqrt\Lambda]^c}\,d\hat\mu(\eta)\leq
\Lambda^{-\alpha/2}\int_{[R-\sqrt\Lambda,R+\sqrt\Lambda]^c}
|\eta-R|^\alpha\,d\hat\mu(\eta)\leq\Lambda^{\frac \alpha 2}
\end{equation}
is small. We can then estimate the velocity field $v$ for $0\leq r
\leq\vartheta\leq \lambda$:
\begin{align*}
\hat{v}(r)&=\int_0^\infty\omega(r,\eta)\,d\hat\mu(\eta)=\int_0^\infty\left[\omega(0,\eta)+r\partial_1\omega(\theta,\eta)\right]\,d\hat\mu(\eta)\\
&=r\int_{[R-\sqrt\Lambda,R+\sqrt\Lambda]}\partial_1\omega(\theta,\eta)\,d\hat\mu(\eta)+r\int_{[R-\sqrt\Lambda,R+\sqrt\Lambda]^c}\partial_1\omega(\theta,\eta)d\hat\mu(\eta).
\end{align*}
Note that $\omega(0,\eta)$ is equal to zero by definition. We then use \eqref{as2} to get the following estimate:
\begin{align*}
\left|\int_{[R-\sqrt\Lambda,R+\sqrt\Lambda]^c}\partial_1\omega(\theta,\eta)d\hat\mu(\eta)\right|&\leq \int_{[R-\sqrt\Lambda,R+\sqrt\Lambda]^c}C(1+\eta^\alpha)d\hat\mu(\eta)\\
&\leq C\left(\Lambda^{\frac \alpha
2}+d_\alpha(\hat\mu,\delta_R)^\alpha\right)\leq
C\Lambda^{\alpha/2}.
\end{align*}
Now, if $\Lambda$ is small enough and $r\in
[0,\vartheta]$, then thanks to \eqref{defepsilon} and an argument as in \eqref{gammatheta} we conclude
\begin{equation*}
\hat{v}(r)\geq
r\partial_1\omega(0,R)\left(1-\int_{[R-\sqrt\Lambda,R+\sqrt\Lambda]^c}\,d\hat\mu(\eta)-Cr\Lambda^{1/2}\right)-Cr\Lambda^{\alpha/2}\geq
\frac {\partial_1\omega(0,R)}2r.
\end{equation*}
\medskip

The second inequality \eqref{setineq2} comes directly from assumption \eqref{as3}
and the continuity of $\omega$: for $r\in [\vartheta,R-\vartheta]$
and $\Lambda$ small enough,
\begin{align*}
\hat{v}(r)=&\int_{[R-\sqrt\Lambda,R+\sqrt\Lambda]}\omega(r,\eta)\,d\hat\mu(\eta)+\int_{[R-\sqrt\Lambda,R+\sqrt\Lambda]^c}\omega(r,\eta)\,d\hat\mu(\eta)\\
\geq &\left(\omega(r,R)-\sqrt\Lambda \|\partial_2\omega\|_{L^\infty([\vartheta,R-\vartheta]\times[R-\sqrt\Lambda,R+\sqrt\Lambda] )}\right)\int_{[R-\sqrt\Lambda,R+\sqrt\Lambda]}\,d\hat\mu(\eta)\\
&-C\int_{[R-\sqrt\Lambda,R+\sqrt\Lambda]^c}(1+\eta^\alpha)\,d\hat\mu(\eta)\\
\geq &\frac 12\omega(r,R)-C\Lambda^{\alpha/2},
\end{align*}
where we have used \eqref{defepsilon}. Since $\omega(\cdot,R)>0$
on $(0,R)$, if $\Lambda>0$ is small enough, $v(r)>0$ on
$[\vartheta, R-\vartheta]$.

For the last inequality in \eqref{setineq2}, we can write the
velocity field as
\begin{equation}\label{estv3}
\hat{v}(r)=\int_{[R-\sqrt\Lambda,R+\sqrt\Lambda]}\omega(r,\eta)\,d\hat\mu(\eta)+\int_{[R-\sqrt\Lambda,R+\sqrt\Lambda]^c}\omega(r,\eta)\,d\hat\mu(\eta)
\end{equation}
and estimate the second term of \eqref{estv3} using \eqref{as3}
and \eqref{defepsilon} to obtain
\begin{align*}
\left|\int_{[R-\sqrt\Lambda,R+\sqrt\Lambda]^c}\omega(r,\eta)\,d\hat\mu(\eta)\right|\leq
C(1+r^\alpha)\int_{[R-\sqrt\Lambda,R+\sqrt\Lambda]^c}(1+\eta^\alpha)\,d\hat\mu(\eta)\leq
C\Lambda^{\alpha/2}(1+r^\alpha)\,.
\end{align*}
Let us distinguish two cases. In the set $r\geq \left(\frac
1\lambda\left(\frac 1\lambda-1\right)\right)^{1/\alpha}$ which is
equivalent to $\frac 1\lambda-\lambda r^\alpha\leq -1$, we deduce
that there exists $C_1>0$ such that $\frac 1\lambda-\lambda
r^\alpha\leq -C_1(1+r^\alpha)$. We can then control the first term
of \eqref{estv3} using \eqref{as1} and \eqref{defepsilon} to get
\begin{align*}
\int_{[R-\sqrt\Lambda,R+\sqrt\Lambda]}\omega(r,\eta)\,d\hat\mu(\eta)&\leq \left(\frac 1\lambda-\lambda r^\alpha\right)\int_{[R-\sqrt\Lambda,R+\sqrt\Lambda]}\,d\hat\mu\\
&\leq
-C_1(1+r^\alpha)\left(1-\int_{[R-\sqrt\Lambda,R+\sqrt\Lambda]^c}\,d\hat\mu\right)\leq
-\frac{C_1}2(1+r^\alpha).
\end{align*}
For $r\in I:=\left[R+\vartheta,\left(\frac 1\lambda\left(\frac
1\lambda-1\right)\right)^{1/\alpha}\right]$, we use the assumption
that $\omega(\cdot,R)<0$ on the compact interval $I$. By
continuity of $\omega$, we thus have that for $\Lambda>0$ small
enough and $r\in I$,
$$
\max\left\{\omega(r,\eta);\,r\in I,
\eta\in[R-\sqrt\Lambda,R+\sqrt\Lambda] \right\}:=-C_2 <0 \, ,
$$
and thus,
$$
\int_{[R-\sqrt\Lambda,R+\sqrt\Lambda]}\omega(r,\eta)\,d\hat\mu(\eta)\leq
-C_2\int_{[R-\sqrt\Lambda,R+\sqrt\Lambda]}\,d\hat\mu \leq
-\frac{C_2}{2}\leq-C_3(1+r^\alpha)<0\, ,
$$
for $r\in I$ and $\Lambda$ small enough using \eqref{defepsilon}.
Then, \eqref{estv3} becomes
\begin{equation*}
\hat{v}(r)\leq
\left(-\min(C_1,C_3)+C\sqrt\Lambda\right)(1+r^\alpha)\leq
-C_4(1+r^\alpha),
\end{equation*}
for any $r\geq R+\vartheta$, if $\Lambda>0$ is small enough.

\

\emph{Step 2.- ``Claim: If $\mu_0$ is close enough to $\delta_R$,
then $\hat \mu_t$ satisfies \eqref{setineq2} at all times.''} Let
$\varphi(t,\xi)$ the associated pseudo-inverse function associated
to $\hat\mu_t$ by \eqref{pseudoinverse}.
We assume that $\hat\mu_0$ satisfies
\begin{equation*}
d_\alpha(\delta_R,\hat\mu_0)<\varepsilon.
\end{equation*}
For any $\varepsilon>0$,
we can estimate $|\varphi(0,\sqrt \varepsilon)-R|$ as follows:
\begin{align*}
\varepsilon & \geq d_\alpha(\delta_R,\hat\mu_0)\geq d_1(\delta_R,\hat\mu_0)=\int_0^1|\varphi(0,\xi)-R|\,d\xi\\
&\geq \max\left(\int_0^{\sqrt
\varepsilon}|\varphi(0,\xi)-R|\,d\xi,\int_{\sqrt
\varepsilon}^1|\varphi(0,\xi)-R|\,d\xi\right)\, .
\end{align*}
Since $\varphi$ is not decreasing, if $\varphi(0,\sqrt \varepsilon)\leq
R$, then $|\varphi(0,\xi)-R|\geq |\varphi(0,\sqrt\varepsilon)-R|$
for $\xi\in[0,\sqrt\varepsilon]$. If $\varphi(0,\sqrt
\varepsilon)\geq R$, then $|\varphi(0,\xi)-R|\geq
|\varphi(0,\sqrt\varepsilon)-R|$ for $\xi\in[\sqrt\varepsilon,1]$,
so that
\begin{equation*}
\varepsilon\geq \min \left(\sqrt\varepsilon |\varphi(0,\sqrt
\varepsilon)-R|,(1-\sqrt\varepsilon) |\varphi(0,\sqrt
\varepsilon)-R|\right),
\end{equation*}
which provides the estimate  $|\varphi(0,\sqrt \varepsilon)-R|\leq
\sqrt\varepsilon$. Similarly, $|\varphi(0,1-\sqrt
\varepsilon)-R|\leq \sqrt\varepsilon$, so that
\begin{equation}\label{ue}
\varphi(0,[\sqrt\varepsilon,1-\sqrt\varepsilon])\subset [R-
\sqrt\varepsilon ,R+ \sqrt\varepsilon].
\end{equation}
Let us define
$\Gamma_\varepsilon(t):=\varphi(t,1-\sqrt\varepsilon)-\varphi(t,\sqrt\varepsilon)$
and
\begin{equation*}
\Theta_\varepsilon(t):=\frac1{1-2\sqrt\varepsilon}\int_{\sqrt\varepsilon}^{1-\sqrt\varepsilon}\varphi(t,\xi)\,d\xi
-R\, .
\end{equation*}
Notice that for $\xi\in [\sqrt\varepsilon,1-\sqrt\varepsilon]$,
\begin{equation}\label{truc}
|\varphi(t,\xi)-R|\leq
\Gamma_\varepsilon(t)+\Theta_\varepsilon(t).
\end{equation}

For $0<\varepsilon\leq\varepsilon_0\leq \Lambda$, we define
$T:=\min\left\{t\in [0,\tau];\,
\Gamma_{\varepsilon}(t)+|\Theta_{\varepsilon}(t)|\geq
\varepsilon_0\text{ or } \,
d_\alpha(\delta_R,\hat\mu_t)\geq\varepsilon_0\right\}$. Thanks to
\eqref{ue}, $T>0$ by continuity for
$0<\varepsilon<\min(\Lambda,\varepsilon_0)$ small enough. We will
show that there exists $\varepsilon>0$ such that $T=+\infty$.
Assume by contradiction that $T<\infty$. By definition of $T$, we
have $d_\alpha(\delta_R,\hat{\mu}_t)\leq \varepsilon_0\leq
\Lambda$ for $t\in [0,T]$.

Thus, $\hat\mu_t$ satisfies \eqref{setineq2} for $t\in [0,T]$,
$v(t,\cdot)$ is positive on $[0,R-\vartheta]$ and negative on
$[R+\vartheta,\infty)$, and then \eqref{pi2} implies that
$\xi\in[0,1]$, $|\varphi(t,\xi)-R|\leq
\max(|\varphi(0,\xi)-R|,\vartheta)$. In particular, by
\eqref{defepsilon} we get
\begin{equation}\label{Omega}
 \int_{[\sqrt\varepsilon,1-\sqrt\varepsilon]^c}|\varphi(t,\xi)-R|^\alpha\,d\xi\leq \int_{[\sqrt\varepsilon,1-\sqrt\varepsilon]^c}\left[|\varphi(0,\xi)-R|^\alpha+\vartheta^\alpha\right]\,d\xi \leq C\sqrt\varepsilon.
\end{equation}
For $t\in [0,T]$, we deduce that
\begin{align*}
\frac d{dt}\Gamma_{\varepsilon}(t)=& \hat{v}(t,\varphi(t,1-\sqrt\varepsilon))-\hat{v}(t,\varphi(t,\sqrt\varepsilon))\\
=&\int_{[\varphi(t,\sqrt\varepsilon), \varphi(t,1-\sqrt\varepsilon)]} \left[\omega(\varphi(t,1-\sqrt\varepsilon),\eta)-\omega(\varphi(t,\sqrt\varepsilon),\eta)\right]\,d\hat\mu_t(\eta)\\
&+\int_{[\varphi(t,\sqrt\varepsilon),
\varphi(t,1-\sqrt\varepsilon)]^c}
\left[\omega(\varphi(t,1-\sqrt\varepsilon),\eta)-\omega(\varphi(t,\sqrt\varepsilon),\eta)\right]\,d\hat\mu_t(\eta).
\end{align*}
The first term can be estimated as it has been done for
$\Gamma(t)$ in the proof of  Theorem~\ref{stability}. To estimate the
second term, we use \eqref{as3}, \eqref{truc}, and \eqref{Omega}
to conclude that
\begin{align*}
&\left|\int_{[\varphi(t,\sqrt\varepsilon), \varphi(t,1-\sqrt\varepsilon)]^c} \left[\omega(\varphi(t,1-\sqrt\varepsilon),\eta)-\omega(\varphi(t,\sqrt\varepsilon),\eta)\right]\,d\hat\mu_t(\eta)\right|\\
&\quad \leq \frac C\lambda \left[1+\min\left(\varphi(t,\sqrt\varepsilon)^\alpha,\varphi(t,1-\sqrt\varepsilon)^\alpha\right)\right] \int_{[\varphi(t,\sqrt\varepsilon), \varphi(t,1-\sqrt\varepsilon)]^c} (1+\eta^\alpha)\,d\hat\mu_t(\eta)\\
&\quad \leq \frac C\lambda
\left[1+o\left(\Gamma_\varepsilon+|\Theta_\varepsilon|\right)\right]
\left[C\sqrt\varepsilon+ \int_{[\sqrt\varepsilon,1-\sqrt\varepsilon]^c}|\varphi(t,\xi)-R|^\alpha\,d\xi\right]\\
&\quad \leq C \sqrt\varepsilon +
o\left(\Gamma_\varepsilon+|\Theta_\varepsilon|\right).
\end{align*}
The same can be done for $\Theta_{\varepsilon}$, and we obtain
\begin{equation*}
\frac
d{dt}\left(\Gamma_\varepsilon+|\Theta_\varepsilon|\right)(t)\leq\max\left(\partial_1\omega(R,R),\,\frac
d{dR}\omega(R,R)
\right)\left(\Gamma_\varepsilon+|\Theta_\varepsilon|\right)(t)
+o\left(\Gamma_\varepsilon+|\Theta_\varepsilon|\right)+
C\sqrt\varepsilon.
\end{equation*}
As it has been done in the proof of  Theorem~\ref{stability},
$\varepsilon_0$ can be chosen small enough such that this implies
for $t\in [0,T]$, that
\begin{equation*}
\frac
d{dt}\left(\Gamma_\varepsilon+|\Theta_\varepsilon|\right)(t)\leq-\gamma\left(\Gamma_\varepsilon+|\Theta_\varepsilon|\right)(t)+
C\sqrt\varepsilon,
\end{equation*}
where $\gamma:=\frac
12\left|\max\left(\partial_1\omega(R,R),\,\frac d{dR}\omega(R,R)
\right)\right|$. Then, for $t\in [0,T]$,
\begin{equation*}
\left(\Gamma_\varepsilon+|\Theta_\varepsilon|\right)(t)\leq\max\left(\left(\Gamma_\varepsilon+|\Theta_\varepsilon|\right)(0),
\frac C\gamma \sqrt\varepsilon\right)\leq C\sqrt\varepsilon,
\end{equation*}
and
\begin{align*}
d_\alpha(\delta_R,\hat\mu_t)^\alpha\leq& \left[\left(\Gamma_\varepsilon+|\Theta_\varepsilon|\right)(t)\right]^\alpha+\int_{[\sqrt\varepsilon,1-\sqrt\varepsilon]^c}|\varphi(t,\xi)-R|^\alpha\,d\xi\\
\leq&
\max\left(\left(\Gamma_\varepsilon+|\Theta_\varepsilon|\right)(0),
\frac C\gamma \sqrt\varepsilon\right)^\alpha+ C\sqrt\varepsilon
\leq C\sqrt\varepsilon,
\end{align*}
due to \eqref{ue} and \eqref{Omega}.

If $\varepsilon>0$ is small enough, this implies that for $t\in
[0,T]$,
$\left(\Gamma_\varepsilon+|\Theta_\varepsilon|\right)(t)\leq\frac{\varepsilon_0}2$,
and $d_\alpha(\delta_R,\hat\mu(t))\leq \frac{\varepsilon_0}{2}$. By a
contradiction argument similar to the one used in the proof of Theorem
\ref{stability}, this shows that if $\varepsilon>0$ is small
enough, then $T=+\infty$, and \eqref{setineq2} is satisfied at all
times.

\

\emph{Step 3.- ``Claim: Asymptotic convergence of $\hat\mu_t$ to
$\delta_R$:''} Since \eqref{setineq2} is satisfied for all $t\geq
0$, $\hat{v}(t,r)\geq C_1\,r$ on $[0,\vartheta]$, and then
$\varphi(t,\xi)\geq \varphi(0,\xi)e^{C_1t}$ if $\varphi(t,\xi)\leq
\vartheta$, due to \eqref{pi2}. We can thus estimate, using
\eqref{pi3}:
\begin{align}
\int_0^\vartheta(1+r^\alpha)\,d\hat\mu_t(r)\leq&(1+R^\alpha)\int_0^\vartheta\,d\hat\mu_t(r)
=(1+R^\alpha)\int_{\{\xi;\varphi(t,\xi)\leq \vartheta\}}\,d\xi \nonumber\\
\leq& (1+R^\alpha)\int_{\{\xi;\varphi(0,\xi)\leq \vartheta
e^{-C_1t}\}}\,d\xi =(1+R^\alpha)\int_0^{\vartheta
e^{-C_1t}}\,d\hat\mu_0(r).\label{newtag}
\end{align}
Since \eqref{setineq2} is satisfied, we claim that
$\hat{v}(t,r)\geq v_1$ on $[\vartheta,R-\vartheta]$, and then
$\varphi(t,\xi)\in [0,R-\vartheta]$ for $t\geq \frac R{v_1}$
implies that $\varphi(t-\frac R{v_1},\xi)\leq\vartheta$.

To see this, we make use of \eqref{setineq2} to get
$\hat{v}(t,\cdot)\geq 0$ on $[0,R-\vartheta]$. If
$\varphi(t,\xi)\in [0,R-\vartheta]$, $\varphi(\cdot,\xi)$ is thus
increasing on $[0,t]$. If $\varphi(t-\frac R{v_1},\xi)\geq
\vartheta$, then $\varphi([t-\frac
R{v_1},t],\xi)\subset[\vartheta,R-\vartheta]$
\begin{equation*}
R>\varphi(t,\xi)-\varphi(t-\frac R{v_1},\xi)=\int_{t}^{t-\frac
R{v_1}} \hat{v}(\sigma,\varphi(\sigma,\xi))\,d\sigma\geq \frac
R{v_1}v_1,
\end{equation*}
which is absurd, thus $\varphi(t-\frac R{v_1},\xi)\leq \vartheta$
as desired. We can then estimate, for $t\geq \frac R{v_1}$, using
\eqref{pi3},
\begin{align*}
\int_0^{R-\vartheta}(1+r^\alpha)\,d\hat\mu_t(r)&\leq (1+R^\alpha)\int_0^{R-\vartheta}\,d\hat\mu_t(r)=(1+R^\alpha)
\left(\int_{0}^{\vartheta}\,d\hat\mu_t(\xi)+\int_{\vartheta}^{R-\vartheta}\,d\hat\mu_t(\xi) \right)\\
&=(1+R^\alpha)\left( \int_{\{\xi;\varphi(t,\xi)\in[0,\vartheta]\}}\,d\xi+\int_{\vartheta}^{R-\vartheta}d\hat\mu_t(\xi)\right)\\
&\leq (1+R^\alpha)\int_{\{\xi;\varphi(0,\xi)\leq\vartheta e^{-C_1 t}\}}\,d\xi+ (1+R^\alpha)\int_{\{\xi;\varphi(t-\frac R{v_1},\xi)\leq
\vartheta\}}\,d\xi  \\
&= (1+R^\alpha)\int_0^{\vartheta e^{-C_1t}}\,d\hat\mu_0(r) + (1+R^\alpha)\int_0^\vartheta\,d\hat\mu_{t-R/v_1}(r).
\\
\end{align*}
Now, using a similar argument as in \eqref{newtag} since $\varphi(t-\frac{R}{v_1},\xi) \leq \vartheta$ in the last integral, we get
\begin{equation}
\int_0^{R-\vartheta}(1+r^\alpha)\,d\hat\mu_t(r) \leq 2(1+R^\alpha)\int_0^{\vartheta
e^{-C_1(t-R/v_1)}}\,d\hat\mu_0(r).\label{est1}
\end{equation}
Thanks to \eqref{setineq2}, if $\varphi(t,\xi)\geq R+ \vartheta$,
then
$$
\varphi(t,\xi)=\varphi(0,\xi)+\int_0^t
\hat{v}(s,\varphi(s,\xi))\,ds\leq \varphi(0,\xi)-v_1t\, .
$$
In particular, $\varphi(t,\xi)^\alpha\leq \varphi(0,\xi)^\alpha$
and, thanks to \eqref{pi2}, we get
\begin{align}
\int_{R+\vartheta}^\infty (1+r^\alpha)\,d\hat\mu_t(r)&=
\int_{\{\xi;\varphi(t,\xi)\geq
R+\vartheta\}}(1+\varphi(t,\xi)^\alpha)\,d\xi\leq
\int_{\{\xi;\varphi(0,\xi)\geq R+\vartheta+v_1t\}}(1+\varphi(0,\xi)^\alpha)\,d\xi\nonumber\\
&=\int_{R+\vartheta+v_1t}^\infty
(1+r^\alpha)\,d\hat\mu_0(r).\label{est2}
\end{align}
Let $\varepsilon>0$. Thanks to \eqref{est1}, \eqref{est2}, there
exists $\tau\geq 0$ such that for any $t\geq\tau$,
\begin{equation}\label{reste}
\int_{[R-\vartheta,R+\vartheta]^c}(1+r^\alpha)\,d\hat\mu_t(r)\leq
\sqrt\varepsilon.
\end{equation}
Then, in particular, for any $t\geq\tau$,
$\int_0^{R-\vartheta}\,d\hat\mu_t(r)\leq\sqrt\varepsilon$ and
$\int_{R+\vartheta}^\infty \,d\hat\mu_t(r)\leq \sqrt\varepsilon$,
that is $\varphi(t,[\sqrt\varepsilon,1-\sqrt\varepsilon])\subset
[R-\vartheta,R+\vartheta]$ and
$\left(\Gamma_\varepsilon+|\Theta_\varepsilon|\right)(t)\leq
3\vartheta$ due to \eqref{truc}. Now, for $t\geq \tau$, with an
argument similar to the one used in Step 2, we get
\begin{equation}
\frac
d{dt}\left(\Gamma_\varepsilon+|\Theta_\varepsilon|\right)(t)\leq\max\left(\partial_1\omega(R,R),\,\frac
d{dR}\omega(R,R) \right)
\left(\Gamma_\varepsilon+|\Theta_\varepsilon|\right)(t)+o\left(\Gamma_\varepsilon+|\Theta_\varepsilon|\right)+
C\sqrt\varepsilon.\label{eststep3}
\end{equation}
Since for any $t\geq \tau$,
$\varphi(t,\cdot)|_{[\sqrt\varepsilon,1-\sqrt\varepsilon]}$ takes
its values in the compact set $[R-\vartheta,R+\vartheta]$
independent of $\varepsilon$, we can apply an argument similar to
the one used in Step 2. Choose $\vartheta>0$ such that
\eqref{eststep3} implies
\begin{equation*}
\frac
d{dt}\left(\Gamma_\varepsilon+|\Theta_\varepsilon|\right)(t)\leq
-\gamma\left(\Gamma_\varepsilon+|\Theta_\varepsilon|\right)(t)+C\sqrt\varepsilon,
\end{equation*}
and then, there exists some $T\geq \tau$ such that for $t\geq T$,
\begin{equation}\label{estgt}
\left(\Gamma_\varepsilon+|\Theta_\varepsilon|\right)(t)\leq 2\frac
{C\sqrt\varepsilon}\gamma.
\end{equation}
To conclude, we notice that thanks to \eqref{reste} and
\eqref{pi3},
\begin{align*}
\int_{[\sqrt\varepsilon,1-\sqrt\varepsilon]^c}|\varphi(t,\xi)-R|^\alpha\,d\xi&\leq \int_{[\sqrt\varepsilon,1-\sqrt\varepsilon]^c}\max(\vartheta^\alpha,|\varphi(t,\xi)-R|^\alpha)\,d\xi\\
&\leq 2\vartheta^\alpha\sqrt\varepsilon+\int_{\{\xi;\,|\varphi(t,\xi)-R|\geq \vartheta\}}|\varphi(t,\xi)-R|^\alpha\,d\xi\\
&\leq 2\vartheta^\alpha\sqrt\varepsilon+\int_{[R-\vartheta,R+\vartheta]^c}|r-R|^\alpha d\hat\mu_t(r)\\
&\leq C\sqrt\varepsilon,
\end{align*}
which, together with \eqref{estgt}, implies that for $t\geq T$,
\begin{align*}
d_\alpha(\hat\mu_t,\delta_R)^\alpha &=\int_{[\sqrt\varepsilon,1-\sqrt\varepsilon]}|\varphi(t,\xi)-R|^\alpha\,d\xi +\int_{[\sqrt\varepsilon,1-\sqrt\varepsilon]^c}|\varphi(t,\xi)-R|^\alpha\,d\xi \\
&\leq C\sqrt\varepsilon.
\end{align*}
Since this is true for any $\varepsilon>0$, it shows that
$d_\alpha(\hat\mu_t,\delta_R)\to 0$ as $t\to\infty$.
\end{proof}


\section{Existence theory}
\label{sec:5}

Existence and uniqueness of weak solutions for the aggregation
equation in $\prob_2(\real^N)\cap L^p(\real^N)$ have been given in
\cite{Laurent2007,BB,BLR}. Weak measure solutions to the the
Cauchy problem for the aggregation equation \eqref{pdes1} where
given in \cite{CDFLS} under the condition that the potential is
smooth except possibly at the origin, the growth at infinity is no
worse than quadratic, and the singularity at the origin of the
derivative of the potential is not worse that Lipschitz. This
section is aimed to give an existence theory of classical
solutions for the aggregation equation.

We will denote by $\mathcal{W}^{m,p}(\real^N)$, $1\leq p\leq
\infty$ and $m\in \mathbb{N}$, the Sobolev spaces.

\begin{theorem}[Existence of classical solutions]\label{classicalexistence}
Let $W$ satisfy
\begin{equation}
\nabla W\in L^1(\real^N),\quad D^2W\in L^1(\real^N),\quad (\Delta W)_+\in L^\infty(\real^N)
\end{equation}
Then, for any initial data $\rho_0(x)\in
\mathcal{W}^{2,\infty}(\real^N)$, there exist classical solutions
$\rho\in C^1([0,T]\times \real^N)\cap
\mathcal{W}^{1,\infty}_{loc}(\real_+,\mathcal{W}^{1,\infty}(\real^N))$
to \eqref{pdes1}. Moreover, if $\rho_0(x)\in
\mathcal{W}^{\kappa,\infty}(\real^N)$ for $\kappa\in \mathbb{N}$,
$\kappa \geq 2$, then $\rho\in C^1((\real^N)\times[0,T])\cap
\mathcal{W}_{loc}^{\kappa-1,\infty}(\real_+,\mathcal{W}^{\kappa-1}(\real^N))$.
Furthermore, assuming in addition that $\rho_0\in L^1(\real^N)$
with bounded second moment, the solution is unique.
\end{theorem}

\begin{proof}[Proof of the theorem.]
{\it Step 1: A priori estimates}. In this step we assume that the
solution is smooth as needed. This assumption will be removed in
the next step.

We consider first $x\in \real^N$ such that
$\rho(x,t)=\norm{\rho(t)}_{\infty}$. Then, $\nabla \rho(x,t)=0$,
and
\begin{align*}
\partial_t\rho(x,t) &= \nabla\rho(x,t)(\nabla W\ast \rho)(x,t)+\rho(x,t)(\Delta W\ast \rho)(x,t)\\
&\leq \rho(x,t)((\Delta W\ast \rho)_+(x,t))\\
&\leq \norm{(\Delta W)_+}_{L^\infty}\rho(x,t),
\end{align*}
so that
\begin{equation}\label{estimateonrho}
\norm{\rho(t)}_{\infty}\leq \norm{\rho_0}_{\infty}e^{\norm{(\Delta W)_+}_{L^\infty}t}.
\end{equation}
Let now $K\in\mathbb{N}$, $K\leq \kappa$, $i=(i_1,\dots,i_N)\in
\N^N $ and $x\in\real^N$ be such that $\sum_{j=1}^{N}i_j=K$, $
\abs{\partial_{i_1}\dots
\partial_{i_N}\rho(x,t)}=\norm{\partial_{i_1}\dots\partial_{i_N}\rho(t)}_{\infty}$,
and we define
\[
  \norm{D^K\rho}_{p}=\sup\lbrace \norm{\partial_\sigma\rho}_{L^p}, \abs{\sigma}\le K\rbrace,\quad 1\leq p \leq \infty, K\leq \kappa,
\]
where $\sigma=(\sigma_1,\dots,\sigma_N)\in\N^N$. W.l.o.g. we
suppose that $\partial_{i_1}\dots \partial_{i_N}\rho(x,t)\geq 0$
(to change the sign of this term, one just needs to replace the
element $e_{1}$ of the basis of $\real^N$ by $-e_{1}$), and then,
\begin{align*}
\partial_t\partial_{i_1}\dots \partial_{i_N}\rho(x,t)=\, &\partial_{i_1}\dots \partial_{i_N}\nabla_x\cdot(\rho(\nabla W\ast \rho))(x,t)\\
=\, & \sum_{k=0}^{K}\sum_{\substack{ |\sigma|=k \\ \sigma\leq i}} (\partial_\sigma \nabla\rho(x,t))\cdot(\partial_{\sigma^c}(\nabla W\ast \rho)(x,t))\\
&\, + \sum_{k=0}^{K}\sum_{\substack{ |\sigma|=k \\ \sigma\leq i }} (\partial_\sigma \rho(x,t)) ( \partial_{\sigma^c}(\Delta W\ast \rho)(x,t)).
\end{align*}
Here, $\sigma\leq i$ denotes $\sigma_j\leq i_j$ for $j=1,\dots, N$ and $\sigma^c=\sigma-i$.
Using that the term $k=K$ in the first sum is zero one obtains
\begin{align}
\partial_t\partial_{i_1}\dots\partial_{i_N}\rho(x,t)=& \nabla \rho(x,t)\cdot(\partial_{i_1}\nabla W)\ast(\partial_{i_2}\dots\partial_{i_N}\rho(x,t))\nonumber\\
&+\rho(x,t)(\Delta W\ast(\partial_{i_1}\dots\partial_{i_N}\rho(x,t))(x,t)\nonumber\\
&+(\partial_{i_1}\dots\partial_{i_N}\rho(x,t))(\Delta W\ast\rho)(x,t)\nonumber\\
&+\sum_{k=1}^{K-1}\sum_{\substack{|\sigma|=k\\\sigma\leq i}}(\partial_{\sigma}\nabla\rho(x,t))(\nabla W\ast\partial_{\sigma^c}\rho)(x,t)\nonumber\\
&+\sum_{k=1}^{K-1}\sum_{\substack{|\sigma|=k\\\sigma\leq i}}(\partial_{\sigma}\rho(x,t))(\nabla W\ast\nabla\partial_{\sigma^c}\rho)(x,t),\label{derivatives}
\end{align}
and then, we get the estimate:
\begin{align*}
\partial_t\partial_{i_1}\dots\partial_{i_N}\rho(x,t)\leq\, &\, \norm{\nabla\rho}_{\infty}\norm{\partial_{i_1}\nabla W}_{L^1}\norm{\partial_{i_2}\dots\partial_{i_N}\rho}_{\infty}\\
&+\norm{\partial_{i_1}\dots\partial_{i_N}\rho}_{\infty}\norm{\Delta W}_{L^1}\norm{\rho}_{\infty}\\
&+\sum_{k=1}^{K-1}\sum_{\substack{|\sigma|=k\\\sigma\leq i}} \norm{\partial_{\sigma}\nabla\rho}_{\infty}\norm{\nabla W}_{L^1}\norm{\partial_{\sigma^c}\rho}_{\infty}\\
&+\sum_{k=1}^{K-1}\sum_{\substack{|\sigma|=k\\\sigma\leq i}} \norm{\partial_{\sigma}\rho}_{\infty}\norm{\nabla W}_{L^1}\norm{\nabla \partial_{\sigma^c}\rho}_{\infty}\\
\leq\, &\, \norm{\rho}_{\infty}\norm{\Delta W}_{L^1}\norm{\partial_{i_1}\dots\partial_{i_N}\rho}_{\infty}+\norm{\partial_{i_1}\nabla W}_{L^1}\norm{\partial_{i_2}\dots\partial_{i_N}\rho}_{\infty}\norm{\nabla\rho}_{\infty}\\
&+C_K\norm{\nabla W}_{L^1}\norm{D^{K-1}\rho}^2_{\infty}\norm{D^K\rho}_{\infty}.
\end{align*}
An induction scheme on $K$ initialized by \eqref{estimateonrho}
then provides the following exponential control for $K\in\lbrace
0,\dots,\kappa\rbrace$:
\begin{equation}
\norm{D^K\rho}_{\infty}\leq C_{1,K}e^{C_{2,K}t}\label{estunifbound}
\end{equation}
where $C_{1,K}, C_{2,K}$ only depend on $K$, $\norm{D^2W}_{L^1}$,
$\norm{(\Delta W)_+}_{L^\infty}$, and $\norm{\rho_0}_{{\mathcal
W}^{\kappa,\infty}}$. Coming back to \eqref{derivatives} the
following estimate on the time derivative follows:
\begin{equation}\label{estunifbound2}
\norm{\frac{d}{dt}D^K\rho}_{\infty}\leq \tilde{C}_{1,K}e^{\tilde{C}_{2,K}t}.
\end{equation}
Finally, taking the derivative in $t$ on \eqref{pdes1} we get
\begin{align*}
   \partial^2_t \rho = &(\nabla W\ast\partial_t \rho)\cdot(\nabla W\ast\rho) +\nabla_x\rho\cdot(\nabla W\ast \partial_t\rho)\\
&+\rho(\Delta W\ast \partial_t \rho) + \partial_t\rho(\Delta W\ast
\rho),
\end{align*}
from which it is easy to derive
\begin{equation}\label{estimatestime}
\norm{\frac{d^2}{dt^2}\rho}_{\infty}\leq Ce^{Ct}\, .
\end{equation}

{\it Step 2: Construction of a solution trough an approximation problem}. Let $W^\varepsilon$ be a smooth approximation of $W$, that is
 $W^{\varepsilon}\in \mathcal{W}^{2,\infty}(\real^N)$ such that
$\norm{D^2W^\varepsilon}_{L^1}$,
$\norm{(\Delta W^\varepsilon)_+}_{L^\infty}$ are uniformly bounded, and:
\[
   \nabla W^\varepsilon\mathop{\longrightarrow}^{\varepsilon\to 0}\nabla W\quad\text{in}\quad L^1(\real^N).
\]
Thanks to \cite{Laurent2007,BLR}, there exists a classical
solution $\rho^\varepsilon\in
\mathcal{W}^{1,\infty}_{loc}(\real_{+},\mathcal{W}^{1,\infty}(\real^N))$
with initial data $\rho_0$ for each regular interaction potential
$W^\varepsilon$.

The estimate \eqref{estunifbound} provides a uniform bound on
$\norm{\rho^\varepsilon}_{\infty}$,
$\norm{\nabla\rho^\varepsilon}_{\infty}$ and
$\norm{D^2\rho^\varepsilon}_{\infty}$. Since $\kappa\geq 2$ then
\eqref{estunifbound2} implies that $\frac{d}{dt}\rho^\varepsilon$
and $\left(\frac{d}{dt}\nabla\rho^\varepsilon\right)$ are
uniformly bounded for $\varepsilon>0$. Applying the
Ascoli-Arzel\'a theorem, due to \eqref{estunifbound},
\eqref{estunifbound2} and \eqref{estimatestime} there exist limits
for $\rho^\varepsilon$, $\partial_t\rho^\varepsilon$ and
$\nabla_x\rho^\varepsilon$ (where we have written $\varepsilon$
instead of $\varepsilon_k$) on $C([0,T]\times B)$ for any compact
subset $B\subset \real^N$ and moreover the limits denoted by
$\rho$, $\partial_t\rho$ and $\nabla_x\rho$ belong to
$C([0,T];\mathcal{W}^{1,\infty}(\real^N))$. For the velocity field
$v^\varepsilon$ we have that
\begin{align*}
\abs{v^\varepsilon (x,t)-v(x,t)}&\leq \int_{\real^N}|(\nabla W^\varepsilon-\nabla W)(x-y)|\rho^\varepsilon(y)\,dy
+\int_{\real^N}\abs{\nabla W(x-y)}\abs{\rho^\varepsilon(y)-\rho(y)}\,dy\\
&=(I)_{\varepsilon}+(II)_{\varepsilon}.
\end{align*}
For $(I)_{\varepsilon}$ one observes that
\[
  \abs{(I)_{\varepsilon}}\leq \norm{\nabla W^\varepsilon-\nabla W}_{L^1}\norm{\rho^\varepsilon}_{\infty} \to 0\quad\text{as}\quad \varepsilon\to 0,
\]
and for $(II)_{\varepsilon}$ one has that that
\[
   \abs{\nabla W(x-y)}\abs{\rho^\varepsilon(y)-\rho(y)}\leq C \abs{\nabla W(x-y)}\in L^1(\real^N).
\]
Then, by the dominated convergence theorem, $(II)_{\varepsilon}\to
0$ as $\varepsilon\to 0$ and, as a consequence, $v^\varepsilon(x,t)$
converges pointwise to $v(x,t)$ in $\real^N\times[0,T]$ for all
$T>0$. The same reasoning is used to prove that $\nabla_x\cdot
v^\varepsilon \to \nabla_x\cdot v$ as $\varepsilon\to 0$. Thus,
the regularized equation
\begin{equation*}
\frac{d}{dt}\rho^\varepsilon=\nabla \rho^\varepsilon\cdot(\nabla W^\varepsilon\ast\rho^\varepsilon)+(\Delta W^\varepsilon\ast\rho^\varepsilon)\rho^\varepsilon)
\end{equation*}
passes to the limit and $\rho\in C^1([0,T]\times\real^N)\cap
\mathcal{W}^{1,\infty}([0,T],\mathcal{W}^{1,\infty}(\real^N))$.

The propagation of the regularity follows from estimates
\eqref{estunifbound} and \eqref{estunifbound2}. The proof of
uniqueness follows from \cite{MR2648318,BLR}.
\end{proof}

\begin{remark}
Under the assumptions on $W$ in the previous theorem we have that $\rho$
is Lipschitz continuous both in space and time, and then the
characteristics are well defined:
\[
   \frac{d}{dt}X_t = -(\nabla W\ast \rho)(X_t,t),
\]
and the solution $\rho$ is given by
\[
   \rho(x,t) = \rho_0(X_{t}^{-1})\det (DX_t^{-1}).
\]
\end{remark}

\begin{remark}
If $W$ and $\rho_0$ are radially symmetric, then one can easily
check that the problem is invariant through rotations around the
origin. The uniqueness result then shows that the solution $\rho$
is radially symmetric at all times. We also point out that if the
solution is compactly supported then it remains of compact support
for all times.
\end{remark}


\section{The example of power law repulsive-attractive potentials}
\label{sec:6}

The aim of this section is to show an example of how to apply the
general instability and stability theory in the case of power law
repulsive-attractive potentials:
\begin{equation} \label{power}
W(x)=\frac{\abs{x}^a }{a }-\frac{\abs{x}^b }{b }
\qquad 2-N<b <a.
\end{equation}
The condition $b<a$ ensures that the potential is repulsive in the
short range and attractive in the long range. One can easily check
that for these type of potentials $\Delta W \in
L^\infty_{loc}(\real^N)$. The condition $2-N<b$ ensures that the
potential is in $\mathcal{W}^{1,q}_{loc}(\real^N)$ for some
$1<q<\infty$. Using algebraic computations, involving the Beta
function, we give the conditions that the powers $a$ and $b$
should satisfy in order to apply the stability and instability
theory, and we construct the bifurcation diagram for these powers.
The main results of this section are the following:

\begin{theorem}[Global existence of solutions for repulsive-attractive potentials]\label{theorem:globalpowers}
Given $W$ by \eqref{power}. Assume $\rho_0\in
\mathcal{W}^{2,\infty}(\real^N)$ is compactly supported and
radially symmetric. Then there exists a global in time classical
solution for \eqref{pde1r}-\eqref{pde2r}. Furthermore, the
solution is compactly supported and confined in a large ball for
all times.
\end{theorem}

\begin{theorem}[Sharp radial stability-instability for spherical shells]\label{theorem:instabilitypowers}
Assume that $W$ is a power law potential as in \eqref{power}.
Then, there exists a unique $R_{ab}>0$ given by
\[
   R_{ab}=\frac{1}{2}\left(\frac{\beta\p{\frac{b+N-1}{2},\frac{N-1}{2}}}{\beta\p{\frac{a+N-1}{2},\frac{N-1}{2}}}   \right)^{\frac{1}{a-b}}
\]
such that $\delta_{R_{ab}}$ is stationary solution to
\eqref{pdes1}. Moreover, the following properties hold:
\begin{enumerate}[(i)]
\item If $2-N<b\leq 3-N$ then $\omega\in C(\real_+^2)\cap
C^1(\real_+^2\setminus \D)$ and  for all $(R,R)\in {\mathcal D}$
we have
 $$
 \lim_{ \substack{(r,\eta)\to (R,R)\\(r,\eta) \notin {\mathcal D}}}  \frac{\partial \omega}{\partial r}(r,\eta)=+\infty.
 $$\label{enumerate:firstpoint}

\item If $b\in \p{3-N, \frac {3a-Na-10+7N-N^2}{a+N-3} }$ then $\omega$ is $C^1(\real_+^2)$ and
$$\partial_1\omega(R_{ab},R_{ab})>0.$$
\label{enumerate:secondpoint}

\item If $b\in \p{\frac {3a-Na-10+7N-N^2}{a+N-3},a }$ then $\omega$ is $C^1(\real_+^2)$ and
\begin{equation*}
\partial_1\omega(R_{ab},R_{ab})<0\quad \text{and} \quad(\partial_1\omega+\partial_2\omega)(R_{ab},R_{ab})<0.
\end{equation*}\label{enumerate:thirdpoint}
\end{enumerate}
As a consequence, if $b\in\p{2-N,\frac {3a-Na-10+7N-N^2}{a+N-3}}$
then $\delta_{R_{ab}}$ is unstable in the sense of {\rm
Theorem~\ref{fat-inst1}} and if $b\in\p{\frac
{3a-Na-10+7N-N^2}{a+N-3},a}$ then $\delta_{R_{ab}}$ is stable in
the sense of {\rm Theorem~\ref{stability2}}.
\end{theorem}
\begin{remark}
Note that indeed  for $3-N<a$ we have
$$
3-N<\frac {3a-Na-10+7N-N^2}{a+N-3}<a.
$$
\end{remark}
\begin{remark}
In \cite{KSUB} the authors study the dynamic of a curve evolving
in $\real^2$ according to the aggregation equation. They perform a
linear stability analysis of the spherical shell steady state. They
consider not only radially symmetric perturbations but also
perturbation which break the symmetry of the spherical shell. The
mode $m=\infty$ corresponds to a perturbation which preserve the
symmetry of the spherical shell.  Using a computation involving the
Gamma function, they show that the mode $m=\infty$ is stable if
and only if $(a-1)(b-1)>1$. In order to prove {\rm (ii)} we will
perform similar type of computations involving special functions.
Note that, when $N=2$
\[
  \frac {3a-Na-10+7N-N^2}{a+N-3}=\frac{a}{a-1}
\]
so \eqref{enumerate:secondpoint} is equivalent to  $(b-1)(a-1)>1$
and we recover the  condition derived in \cite{KSUB}.
\end{remark}
As a summary of all the stability and instability results for a
$\delta_{R_{ab}}$ stationary states for power law potentials we
show the bifurcation diagram in Figure~\ref{fig:planes}. For
powers inside the region between $b=2-N$ and the curve one has
instability of the $\delta_{R_{ab}}$. In the region above the
curve one has stability of the $\delta_{R_{ab}}$.

\begin{figure}[]
\scalebox{0.35}{\includegraphics{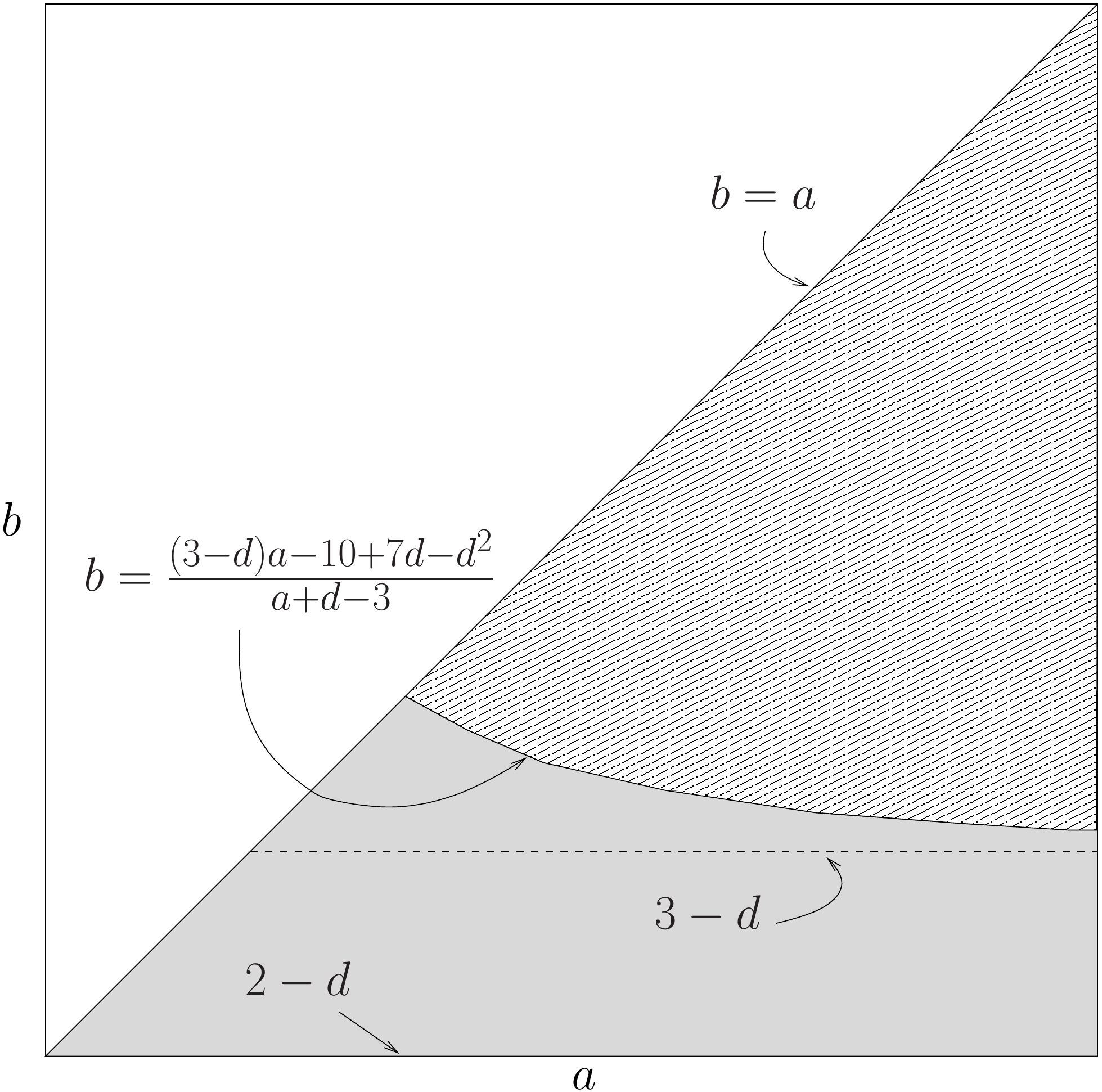}}
\caption{Bifurcation diagram. Stability and instability regions
for $\delta_{R_{ab}}$.} \label{fig:planes}
\end{figure}

\subsection{Proof of Theorem~\ref{theorem:globalpowers}}
In order to prove the global existence theorem we need to
introduce some notations. For potentials defined by \eqref{power},
the kernel $\omega(r,\eta)$ defined in \eqref{omega-def2} becomes
\begin{gather}
\label{omega}
\omega(r,\eta)= r^{b -1} \psi_b (\eta/r)-r^{a -1} \psi_a (\eta/r) \\
\psi_a(s)= \frac{1}{\sigma_N}\int_{\partial B(0,s)} \frac{(e_1-sy)\cdot e_1}{\abs{e_1-s y}^{2-a}}  \;  d \sigma(y) \label{psi_a}
\end{gather}
The properties of the function $\psi_a(s)$ that we need are
summarized in the following lemma and can be found in \cite{DONG}.

\begin{lemma}[Properties of the function $\psi_a(s)$]\strut\label{lemma1}
The function $\psi_a$ is continuous with $\psi_a(0)=1$ and
$\lim_{s\to\infty} s^{2-a}\psi_a(s)=\frac{N+a-2}{N}$.
\end{lemma}

The main difficulty we have to cope with is the growth at infinity
of the attractive part which restricts the range of direct
application of Theorem \ref{classicalexistence}.

\begin{proof}[Proof of Theorem~\ref{theorem:globalpowers}]
Due to translational invariance we can assume without loss of
generality that the center of mass is located at zero. We write
$W$ as $W(x)=W_R(x)+W_A(x)$ where $W_A(x)=\frac{|x|^a}{a}$ is the
attractive part and $W_R(x)=-\frac{|x|^b}{b}$ is the repulsive
part. In addition, since $W$ is radially symmetric, that is
$W(x)=k(|x|)$ then we define $k(r):=k_R(r)+k_A(r)$. Finally, we
write $\omega_A(r,\eta)=r^{a-1}\psi_a(\eta/r)$ and
$\omega_R(r,\eta)=r^{b-1}\psi_b(\eta/r)$ with
$\omega(r,\eta)=\omega_R(r,\eta)-\omega_A(r,\eta)$.

\

{\it Step 1: A priori estimates on the support of $\hat \rho$.}
Suppose that $\hat \rho$ is a smooth radially symmetric solution
for the equation \eqref{pdes1} with compactly supported initial
data $\rho_0\in \mathcal{W}^{2,\infty}(\real^N)$. Since the
solutions belong to
$\mathcal{W}_{loc}^{1,\infty}(\real_+,\mathcal{W}^{1,\infty}(\real^N))$
and are compactly supported, the velocity field $\hat v$ is
Lipschitz continuous in time and space and the characteristics are
well defined. Thus $\hat v$ generates a $C^1$ flow map $r(t,r_0)$,
$t\in [0,T]$, $r_0\in \real_+$:
\begin{align*}
\frac{d}{dt}r(t)&=\hat v(t,r(t,r_0)),\\
r(0,r_0)&=r_0.
\end{align*}
Let us define by $r_2(t)$ the characteristic curve starting at point $r_2(0)=\max\lbrace{\rm supp}(\hat \mu_0)\rbrace$.
Then,
\begin{align}\label{support_radial}
\frac{d}{dt}r_2(t)&=\hat v(t,r_2(t))=\int_{0}^{\infty}\omega(r_2(t),\eta)d\hat{\mu}_t(\eta)=\int_{0}^{r_2(t)}\omega(r_2(t),\eta)d\hat{\mu}_t(\eta)\nonumber\\
&=-\int_{0}^{r_2(t)}r_2(t)^{a-1}\psi_a(\eta/r_2(t))d\hat{\mu}_t(\eta)+\int_{0}^{r_2(t)}r_2(t)^{b-1}\psi_b(\eta/r_2(t))d\hat{\mu}_t(\eta),
\end{align}
where we have used the expression of $\omega$ given in \eqref{omega}.
Here $\hat\mu_t$ denotes the measure with density $\hat\rho_t$. Using
 the properties of $\psi_a$ in Lemma~\ref{lemma1} in \eqref{support_radial} we obtain the following inequality:
\begin{align}\label{globalexistenceineq}
\frac{d}{dt}r_2(t)\leq K_b r_2(t)^{b-1}-K_a r_2(t)^{a-1},
\end{align}
\text{where}
\begin{equation*}
\begin{cases}
K_a=1 \qquad\quad\text{and} \quad K_b=\psi_b(1),& \text{if $2\leq b<a$,}\\
K_a=1 \qquad\quad\text{and} \quad K_b=1,& \text{if $2-N< b<2\leq a$,}\\
K_a=\psi_a(1)\quad \,\text{and} \quad K_b=1,& \text{if $2-N< b<a<2$.}
\end{cases}
\end{equation*}
Defining $\tilde
R_{ab}:=\left(\frac{K_a}{K_b}\right)^{\frac{1}{b-a}}$ and
rewriting \eqref{globalexistenceineq} as
\[
   \frac{d}{dt} r_2(t)\leq r_2(t)^{a-1}(K_b r_2(t)^{b-a}-K_a)
\]
one realizes that $r_2(t)\leq \overline{R}:=\max(r_2(0), \tilde
R_{ab})$ which proves that the $\text{supp}(\hat \mu_t)$ is
bounded and contained in $B(0,\overline{R})$ for all times.

\

{\it Step 2: Global existence.} Given $0< \varepsilon <1$,
consider $\chi_{\varepsilon}(r)$ a $C^\infty(0,\infty)$ decreasing
function with $0<\varepsilon<1$ such that
$\chi_{\varepsilon}(r)=1$ if $0<r<1/\varepsilon$ and
$\chi_{\varepsilon}(r)=0$ if $r>1+1/\varepsilon$. Define
$f^\varepsilon(r):=\chi_{\varepsilon}(r)\cdot k_{A}'(r)\in
L^1(0,\infty)$, $k_{A}^{\varepsilon}(r):=\int_{0}^{r}f^\varepsilon(s)ds$ and
$k^{\varepsilon}(r)=k_{R}(r)+k_{A}^{\varepsilon}(r)$. Now, the
potential $W^\varepsilon(x)=k^\varepsilon(|x|)$ satisfies the
hypotheses of Theorem~\ref{classicalexistence} so we have
existence and uniqueness of classical solution
$\hat\rho^\varepsilon$ to \eqref{pde1r}-\eqref{pde2r} in $[0,T]$
with initial data $\hat\rho_0$. We denote by
$\hat\mu_t^\varepsilon$ the measure with density
$\hat\rho_t^\varepsilon$.

Consider $r_2=\max\lbrace {\rm supp}(\hat\mu _0)\rbrace$ and the
characteristic curve $r_2^\varepsilon(t)$ starting at point
$r_2=r_2(0)$. Computing the derivative with respect to time, one
has
\begin{align*}
\frac{d}{dt}r_2^\varepsilon(t)&=\int_{0}^{\infty}\omega^\varepsilon(r_2^\varepsilon(t),\eta)d\hat{\mu}^\varepsilon_t(\eta)=\int_{0}^{r_2^\varepsilon(t)}\omega^\varepsilon(r_2^\varepsilon(t),\eta)d\hat{\mu}^\varepsilon_t(\eta)\\
&=-\int_{0}^{r_2^\varepsilon(t)}\omega_A^\varepsilon(r_2^\varepsilon(t),\eta)d\hat{\mu}^\varepsilon_t(\eta)+\int_{0}^{r_2^\varepsilon(t)}r_2^\varepsilon(t)^{b-1}\psi_b(\eta/r_2^\varepsilon(t))d\hat{\mu}^\varepsilon_t(\eta)\\
&\leq C_b r_2^\varepsilon(t)^{b-1},
\end{align*}
where we have split the kernel $\omega^\varepsilon$ into its
attractive and repulsive parts, and we have used that
$\omega_A^\varepsilon\geq 0$. The constant $C_b$ depends on $b$.
The last inequality leads us to
\begin{equation}\label{inequalityR2}
r_2^\varepsilon(t)\leq \sigma(t):=(C_b (2-b) t + r_2^{2-b})^{\frac{1}{2-b}},
\end{equation}
which says that the solution exists at least up to time $T^\ast:=
\frac{1}{2}\min\left\lbrace
T,T_b\right\rbrace$.
where
$$T_b=
\begin{cases}
\frac{r_2(0)^{2-b}}{C_b(b-2)} & \text{if}\quad b>2,\\
+\infty & \text{if}\quad b\leq 2.
\end{cases}
$$
In addition, \eqref{inequalityR2} gives us a uniform estimate for
the support of $\rho^\varepsilon$ up to time $T^\ast$. Notice that
for all $t\leq T^\ast$ and $\varepsilon>0$ such that
$2\sigma(T^\ast)\varepsilon<1$,  then $\nabla
W^\varepsilon\ast\rho_t^\varepsilon=\nabla W\ast
\rho_t^\varepsilon$ for all $x\in{\rm supp}(\rho^\varepsilon_t)$
and all $t\in [0,T^\ast]$. As a consequence $\omega$ given by
\eqref{omega} and $\omega^\varepsilon$ associated to
$W^\varepsilon$ by \eqref{omega-def2} are equal in the set
$\lbrace(r,\eta)\,|\,(r,\eta)\in {\rm
supp}(\rho_t^\varepsilon)^2\rbrace$ for all $t\leq T^\ast$.
Therefore, we can write:

\begin{align*}
\frac{d}{dt}r_2^\varepsilon(t)&=   -\int_{0}^{r_2^\varepsilon(t)}\omega_A^\varepsilon(r_2^\varepsilon(t),\eta)d\hat{\mu}^\varepsilon_t(\eta)+\int_{0}^{r_2^\varepsilon(t)}r_2^\varepsilon(t)^{b-1}\psi_b(\eta/r_2^\varepsilon(t))d\hat{\mu}^\varepsilon_t(\eta)\\
&=-\int_{0}^{r_2^\varepsilon(t)}\omega_A(r_2^\varepsilon(t),\eta)d\hat{\mu}^\varepsilon_t(\eta)+\int_{0}^{r_2^\varepsilon(t)}r_2^\varepsilon(t)^{b-1}\psi_b(\eta/r_2^\varepsilon(t))d\hat{\mu}^\varepsilon_t(\eta)\\
&=-\int_{0}^{r_2^\varepsilon(t)}r_2^\varepsilon(t)^{a-1}\psi_a(\eta/r_2^\varepsilon(t))d\hat{\mu}^\varepsilon_t(\eta)+\int_{0}^{r_2^\varepsilon(t)}r_2^\varepsilon(t)^{b-1}\psi_b(\eta/r_2^\varepsilon(t))d\hat{\mu}^\varepsilon_t(\eta),
\end{align*}
and we can use the a priori estimates developed in Step 1. Then,
for all $t\leq T^\ast$  we can conclude that
$r_2^\varepsilon(t)\leq \overline{R}$. Now, let us take
$\varepsilon$ such that $2\varepsilon \overline{R}<1$. Therefore
$\nabla W^\varepsilon\ast \rho_t^\varepsilon=\nabla W\ast
\rho_t^\varepsilon$. For all $t\leq T^\ast$ in the support of
$\rho_t^\varepsilon$. By uniqueness $\rho_t^{\varepsilon}=:\rho_t$
for all $2\varepsilon \overline{R}<1$ and it is a classical
solution to \eqref{pdes1} with potential $W$. Summarizing, we have
shown the existence of solution in the time interval $[0,T^\ast]$
with $r_2(T^\ast)\leq \overline{R}$. Now, we can extend and repeat
this argument for a time step $\Delta
t:=\frac{1}{2}\min\p{1,\overline{T}_b}$, where $\overline{T}_b=\frac{\overline{R}^{2-b}}{C_b(b-2)}$ if $b>2$ or $\overline{T}_b=+\infty$ if $b\leq 2$,
obtaining a solution up to time $T^\ast+\Delta t$ such that
$r_2(t)\leq \overline{R}$ for all $t\in [0,T^\ast+\Delta t]$.
Since $\Delta t$ is independent of the initial data and
$\varepsilon >0$, then we can extend the solution for all times.

Finally, the a priori estimates on the support of $\hat\mu$ show
that the support of the solution remains compact for all times.
\end{proof}

\subsection{Proof of the Theorem~\ref{theorem:instabilitypowers}}

It is first convenient to rewrite \eqref{psi_a} as:
\begin{equation}\label{psi_beta}
\psi_a(s)=\frac{\sigma_{N-1}}{\sigma_N}\int_{0}^{\pi}\frac{(1-s\cos\theta)(\sin\theta)^{N-2}}{(1+s^2-2s\cos\theta)^{\frac{2-a}{2}}}d\theta.
\end{equation}

We recall that $\omega(r,\eta)$ is the velocity at $r$ generated
by $\partial B(0,\eta)$. So a $\delta_R$ with $R>0$ is a steady
state if and only if $\omega(R,R)=0$, i.e.
\begin{equation} \label{Rab}
R=R_{ab }= \p{\frac{\psi_b (1)}{\psi_a (1)}}^{\frac{1}{a -b }},
\end{equation}
where we have used \eqref{omega}.

\begin{proof}[Proof of the Theorem~{\rm \ref{theorem:instabilitypowers}}]
The point \eqref{enumerate:firstpoint} is a direct consequence of Lemma~\ref{fat-inst1}. Let us prove  \eqref{enumerate:secondpoint}. From Lemma~\ref{fat-inst1}, see also \cite{DONG} it is clear that  $\omega \in C^1(\real_+^2)$ and we have
\begin{equation} \label{bibi}
\frac{\partial \omega}{ \partial r}(R_{ab },R_{ab })= R_{ab }^{b -2} \Big[ (b -1) \psi_b  (1)-  \psi'_b   (1)\Big]
-R_{ab }^{a -2} \Big[ (a -1) \psi_a  (1)- \psi'_a    (1) \Big]
\end{equation}
After some algebra, one easily get from \eqref{Rab} and \eqref{bibi} that  $\frac{\partial \omega}{ \partial r}(R_{ab },R_{ab })>0$ is equivalent to
\begin{equation}\label{stabilitycondition2}
a -\frac{\psi_{a }'(1)}{\psi_a (1)}<b -\frac{\psi_{b }'(1)}{\psi_b (1)}.
\end{equation}
Both $\psi_a(1)$ and $\psi_a'(1)$ can be expressed in terms of the Beta function. Recall that one of the expressions of the Beta function is:
$$
\beta(x,y)=2\int_0^{\pi/2} (\cos\theta)^{2x-1} (\sin\theta)^{2y-1} d\theta.
$$
We first compute $\psi_a(1)$. Using \eqref{psi_beta}:
\begin{align*}
\frac{\sigma_N}{\sigma_{N-1}}  \psi_{a }(1)&=\int_{0}^{\pi}\frac{(1-\cos\theta)}{A(1,\theta)^{2-a }}(\sin\theta)^{N-2}d\theta
=2^{\frac{a-2 }{2}}\int_{0}^{\pi}(1-\cos\theta)^{a /2}(\sin\theta)^{N-2} d\theta \nonumber\\
&= 2^{\frac{a-2 }{2}}\int_{0}^{\pi}(2 \sin^2\frac{\theta}{2})^{a /2}(2 \cos\frac{\theta}{2} \sin\frac{\theta}{2})^{N-2} d\theta \nonumber\\
&= 2^{a+N-3 }\int_{0}^{\pi}\p{ \sin\frac{\theta}{2}}^{a+N-2} \p{\cos\frac{\theta}{2}}^{N-2} d\theta\nonumber \\
&=2^{a+N-3} \beta\p{\frac{a+N-1}{2},\frac{N-1}{2}}
\end{align*}
where we have used the fact that $A(1,\theta)=\sqrt{2(1-\cos \theta})$ and the identities
$1-\cos\theta=2\sin^2\frac{\theta}{2}$ and
$\sin\theta=2\cos\frac{\theta}{2}\sin\frac{\theta}{2}$.
Similarly we compute
\begin{align*}
\frac{\sigma_N}{\sigma_{N-1}}\frac{N-1}{(a-2)(a+N-2)}   \psi_{a }'(1)&=\int_{0}^{\pi}\frac{(\sin\theta)^N}{A(1,\theta)^{4-a }}\,d\theta
=\int_{0}^{\pi}\frac{(\sin\theta)^N}{(2(1-\cos \theta))^{\frac{4-a}{2} }}d\theta\nonumber\\
&=\int_{0}^{\pi}\frac{(2 \cos \frac{\theta}{2}\sin\frac{\theta}{2})^N}{(2(2\sin^2 \frac{\theta}{2}))^{\frac{4-a}{2} }}d\theta\nonumber\\
&=2^{N+a-4} \int_{0}^{\pi}\p{\cos \frac{\theta}{2}}^N\p{\sin\frac{\theta}{2}}^{N+a-4} d\theta\nonumber\\
&=2^{N+a-4} \beta\p{\frac{a+N-3}{2},\frac{N+1}{2}}
\end{align*}
Note that since $a+N-3>0$  the Beta function is well defined.
If we compute the quotient  we obtain:
\[
   \frac{\psi_{a}'(1)}{\psi_{a}(1)}=\frac{1}{2}\frac{(a -2)(a +N-2)}{N-1}\frac{\beta\left(\frac{a+N-3}{2},\frac{N+1}{2}\right)}{\beta\left(\frac{a+N-1}{2},\frac{N-1}{2}\right)}.
\]
 At this point, we remind that $\beta(z,t)=\frac{\Gamma(z)\Gamma(t)}{\Gamma(z+t)}$. With this expression, the quotient can be simplified as
\[
   \frac{\psi_{a}'(1)}{\psi_{a}(1)}=\frac{1}{2}\frac{(a -2)(a +N-2)}{N-1}\frac{\Gamma\left(\frac{N+1}{2}\right)\Gamma\left(\frac{a+N-3}{2}\right)}{\Gamma\left(\frac{a+N-1}{2}\right)\Gamma\left(\frac{N-1}{2}\right)},
\]
and if we use that $\Gamma(z+1)=z\Gamma(z)$ then, the gamma quotient can be reduced to ${(N-1)}/{(a+N-3)}$ and then we obtain
\begin{equation*}
\frac{\psi_{a}'(1)}{\psi_{a}(1)}=\frac{1}{2}\frac{(a -2)(a +N-2)}{a+N-3}.
\end{equation*}
Plugging the above expression into  \eqref{stabilitycondition2} and doing some algebra we deduce
$$
(a+N-3)b^2+ (N^2-7N+10-a^2)b - (N^2-7N+10)a - (N-3)a^2>0.
$$
The roots of the quadratic form are
$$
b=a \quad \text{and} \quad  b=\frac {3a-Na-10+7N-N^2}{a+N-3}
$$
which gives \eqref{enumerate:secondpoint}. To prove
\eqref{enumerate:thirdpoint} one just need to replace  the $<$
sign by a $>$ in \eqref{stabilitycondition2} to get the first
condition. For the second condition of stability, we can easily
compute $\frac{\partial\omega}{\partial\eta}$ from \eqref{omega}
and \eqref{bibi} to get
$$
\left(\frac{\partial \omega}{ \partial r}+\frac{\partial \omega}{
\partial \eta}\right)(R_{ab },R_{ab })= R_{ab }^{b -2} (b -1) \psi_b
(1) -R_{ab }^{a -2} (a -1) \psi_a  (1) \,.
$$
Now, using the definition of $R_{ab}$ in \eqref{Rab}, we finally
obtain
$$
\left(\frac{\partial \omega}{ \partial r}+\frac{\partial \omega}{
\partial \eta}\right)(R_{ab },R_{ab })= b-a< 0\,.
$$
The stated instability and stability are direct applications of
Theorem~\ref{fat-inst1} and Theorem~\ref{stability2} respectively.
\end{proof}

\begin{remark}
We want to point out that the general theory developed in the
previous sections is still working in dimension $N=1$ for even
solutions which correspond to the radially symmetric solutions in
higher dimensions. In the case $N=1$, and for even solutions, the
function $\psi_a$, corresponding to $W(x)=\frac{\abs{x}^a}{a}$ reads
\begin{equation*}
\psi_a(s)=\frac{1}{2}\left[ \p{1-s}\abs{1-s}^{a-1}+\p{1+s}\abs{1+s}^{a-1}\right].
\end{equation*}
One can easily check that the properties of the function $\psi_a$
and $\omega$ for $W(x)=\frac{\abs{x}^a}{a}-\frac{\abs{x}^b}{b}$ in
$N=1$ are the same as in {\rm Lemma~\ref{lemma1}} and in {\rm
Theorem~\ref{theorem:instabilitypowers}}. The radius is
$R_{ab}=\frac{1}{2}$ whatever the powers are, see \eqref{Rab}.
{\rm Theorem~\ref{theorem:instabilitypowers}} applies: if $b\in
(1,2)$ then we are in the instability case
\eqref{enumerate:firstpoint} and if $b\in [2,a)$ we are in the
stability case. The curve which separates the instability and
stability regions in Figure~\ref{fig:planes} degenerates and
becomes the line $b=2$.

Moreover, in \cite{FellnerRaoul1,FellnerRaoul2} the authors proved
the existence of weak solutions and convergence, up to extractions
of subsequences, of $\rho(\cdot,t)$ for potentials like
$W(x)=\frac{\abs{x}^2}{2}-\frac{\abs{x}^b}{b}$, $b\in (0,1]$. They
also showed numerical simulations supporting the conjecture that
the stationary state
$\frac{1}{2}\p{\delta_{x=-1/2}+\delta_{x=1/2}}$ is unstable. Note
that our {\rm Theorem~\ref{theorem:instabilitypowers}}, only
applies in $N=1$ for $1<b<a$. Since global existence was proven in
\cite{FellnerRaoul1,FellnerRaoul2}, then our instability result
also applies in those cases. Summarizing, we can include $a=2$,
$b\in (0,1]$ in the instability regions using {\rm
Theorem~\ref{theorem:instabilitypowers}}.
\end{remark}


\section{Numerical results}
\label{sec:7}

In this section, we illustrate the previous results and get some
further conjectures for the instability cases. Our numerical code
is based on the inverse distribution function in radial
coordinates. As it was reminded in \eqref{pi2}, the equation for
the inverse distribution function reads
\begin{equation}\label{radialinverse}
\frac{\partial \varphi}{\partial
t}(t,\xi)=\int_{0}^{1}\omega(\varphi(t,\xi),\varphi(t,\tilde\xi))d\tilde\xi.
\end{equation}
A solution of \eqref{pde1r} converges to a Dirac mass if and only
if its pseudo inverse distribution becomes flat.

Numerical codes based on  \eqref{radialinverse} are then more
stable when dealing with mass concentration. We will then use a
backward Euler scheme in time coupled to a composite Simpson rule
to approximate the integral term, and solve the resulting
nonlinear system by the Newton-Raphson algorithm. Let us remark
that the convergence of the semi-discrete backward Euler scheme is
equivalent to the convergence of the JKO variational scheme for
\eqref{pdes1} (see \cite{BCC,MR2566595}). The convergence of the
semi-discrete backward Euler scheme is therefore known under
suitable conditions on the interaction potential, see \cite{CDFLS}
for details. All simulations are done for $N=2$.

\subsection*{{\bf Test case, total concentration at the origin: $W(x)=\frac{|x|^2}{2}$}}

In this case, \eqref{radialinverse} reduces to $\frac{\partial
\varphi}{\partial t}(t,\xi)+\varphi(t,\xi)=0$. To test our scheme,
we use this attractive potential for which the solution converges
exponentially fast to a total concentration at zero, that is to
$\bar\varphi\equiv 0$. See Figure~\ref{fig:onlyattractive}.

\begin{figure}[h]
\includegraphics[scale=0.55]{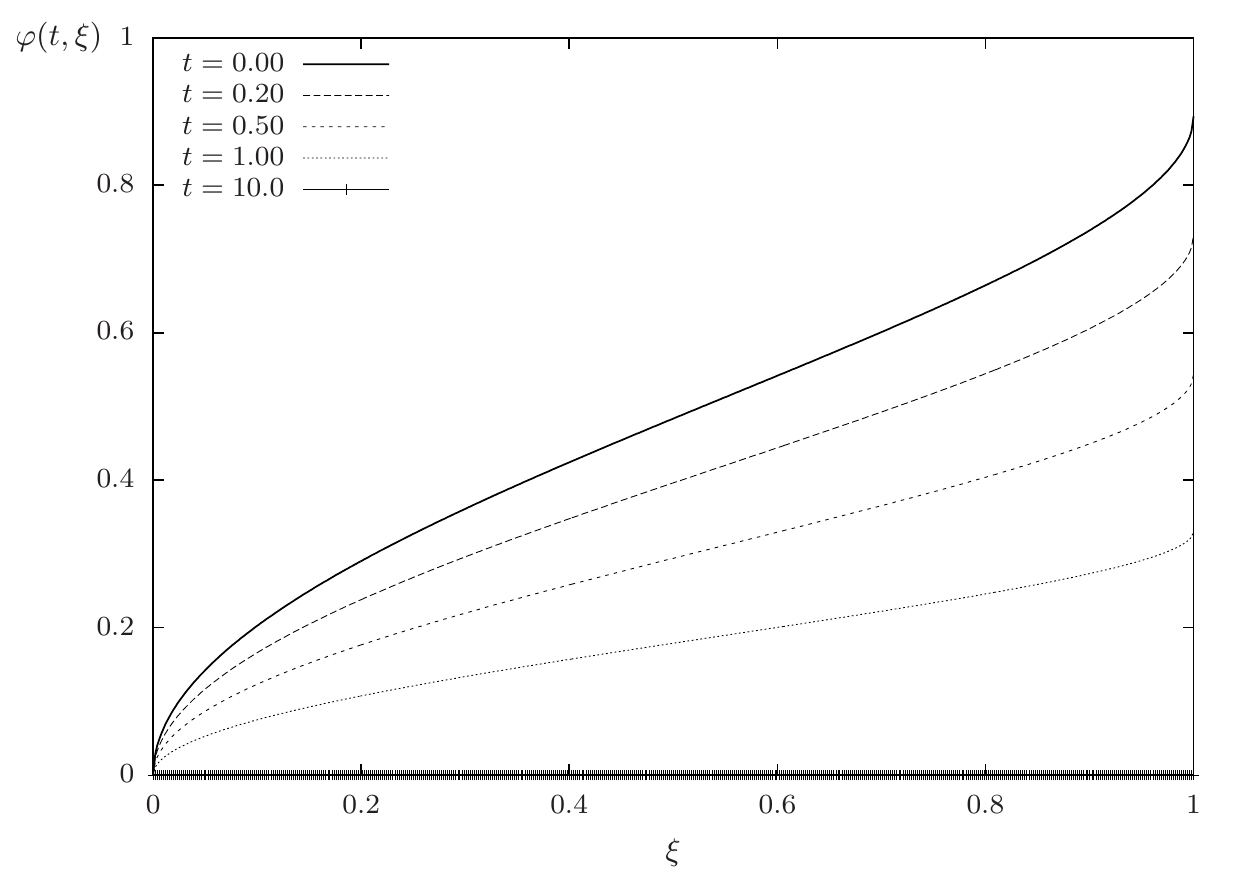}
\caption{Evolution of $\xi\mapsto\varphi(t,\xi)$ for
$W(x)=\frac{|x|^2}{2}$ towards total concentration at $0$.}
\label{fig:onlyattractive}
\end{figure}

\subsection*{Stability Case for the Spherical Shell: $W(x)=\frac{|x|^4}{4}-\frac{|x|^2}{2}$.}

In this case, we have an repulsive-attractive power law potential
with powers in the stability region of Figure~\ref{fig:planes}. We
thus expect that the mass will concentrate towards a spherical
shell, thanks to the results of
Theorem~\ref{theorem:instabilitypowers}. The radius of the
spherical shell can be computed using \eqref{Rab}:
\[
 R_{ab}=\left(\frac{\psi_2(1)}{\psi_4(1)}\right)^{\frac{1}{2}}=\frac{\sqrt{3}}{3}.
\]
For $b\geq 2$ and both $a$ and $b$ integers, one can compute
explicitly the expression for the velocity field $\omega(r,\eta)$,
which is a polynomial function, in our case
$\omega(r,\eta)=-r^3-2r\eta^2+r$. The evolution of $\varphi$ is
shown in Figure~\ref{fig:phiregular}. In
Figure~\ref{fig:phiregular} we also plot the velocity field
$r\mapsto\omega(r,R_{ab})$. Notice that $r\mapsto\omega(r,R_{ab})$
satisfies the conditions of Theorems~\ref{stability} and
\ref{stability2}: $\omega(R_{ab},R_{ab})=0$,
$\partial_1\omega(R_{ab},R_{ab})<0$,
$\textrm{sign}(\omega(r,R_{ab}))=\textrm{sign}(R_{ab}-r)$,
$\partial_1\omega(0,R_{ab})>0$.

\begin{figure}[h]
\includegraphics[scale=0.55]{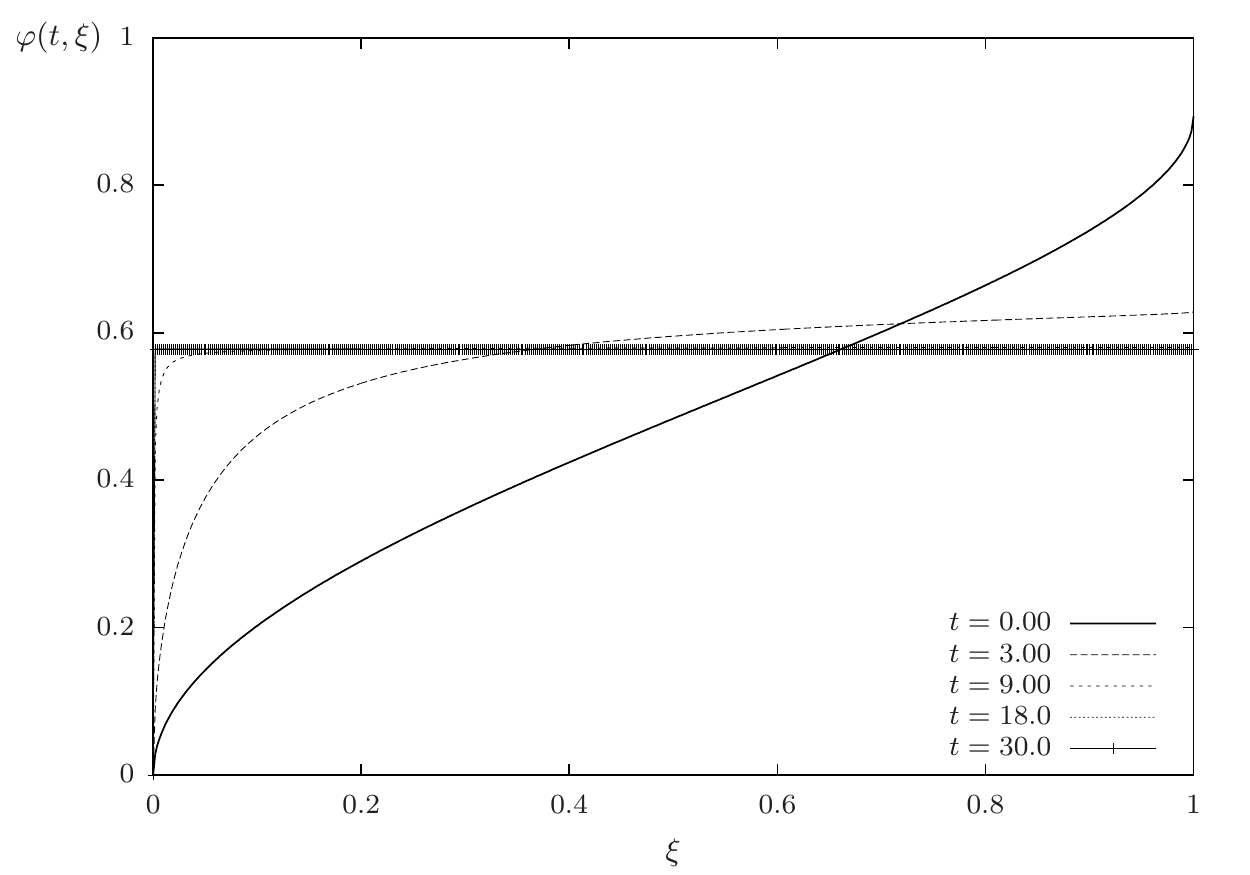}
\includegraphics[scale=0.55]{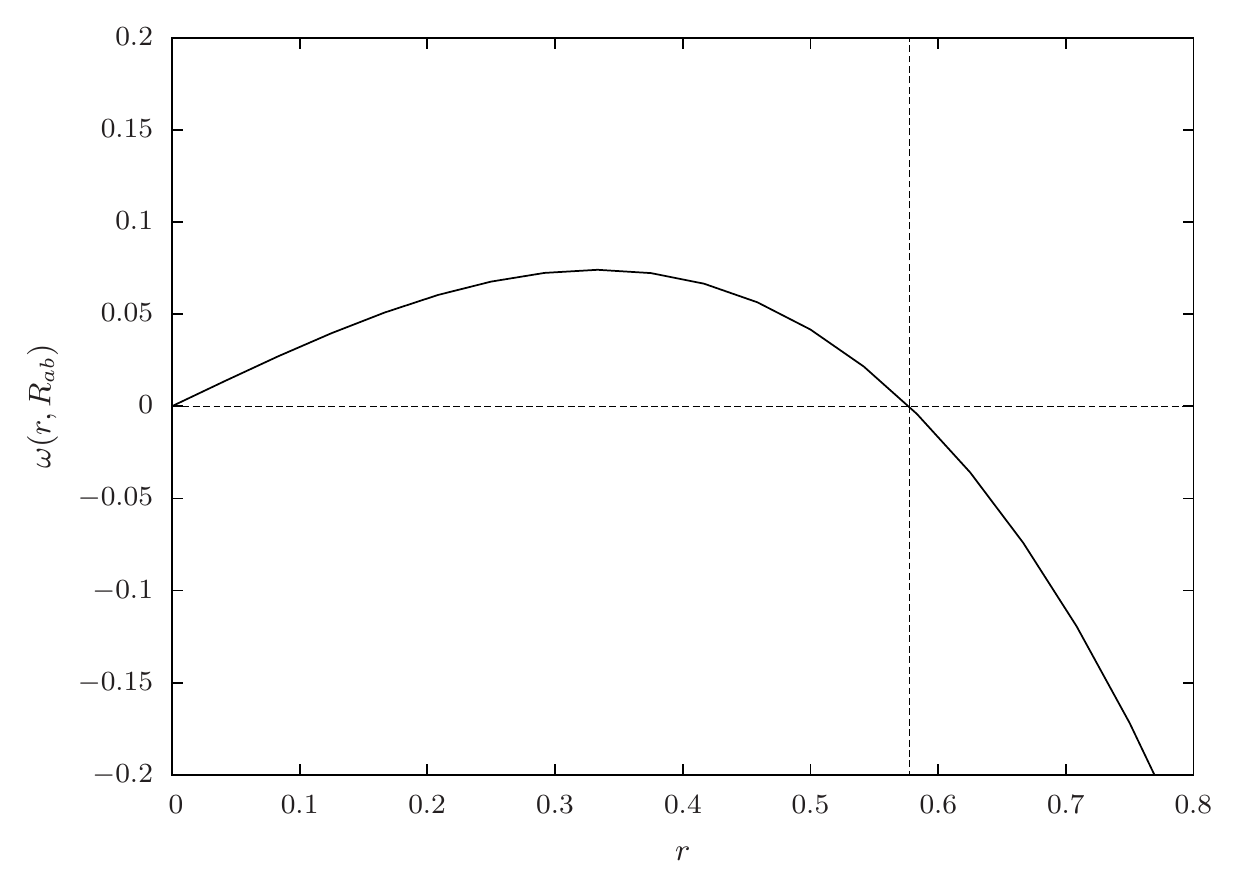}
\caption{Case $W(x)=\frac{|x|^4}{4}-\frac{|x|^2}{2}$. Left:
Evolution of $\xi\mapsto\varphi(t,\xi)$ towards the uniform
distribution on the sphere of radius $R_{ab}=\frac{\sqrt{3}}{3}$.
Right: Velocity field $r\mapsto\omega(r,\frac{\sqrt{3}}{3})$ with
the vertical line pointing out $R_{ab}$.} \label{fig:phiregular}
\end{figure}

\begin{figure}[h]
\includegraphics[scale=0.55]{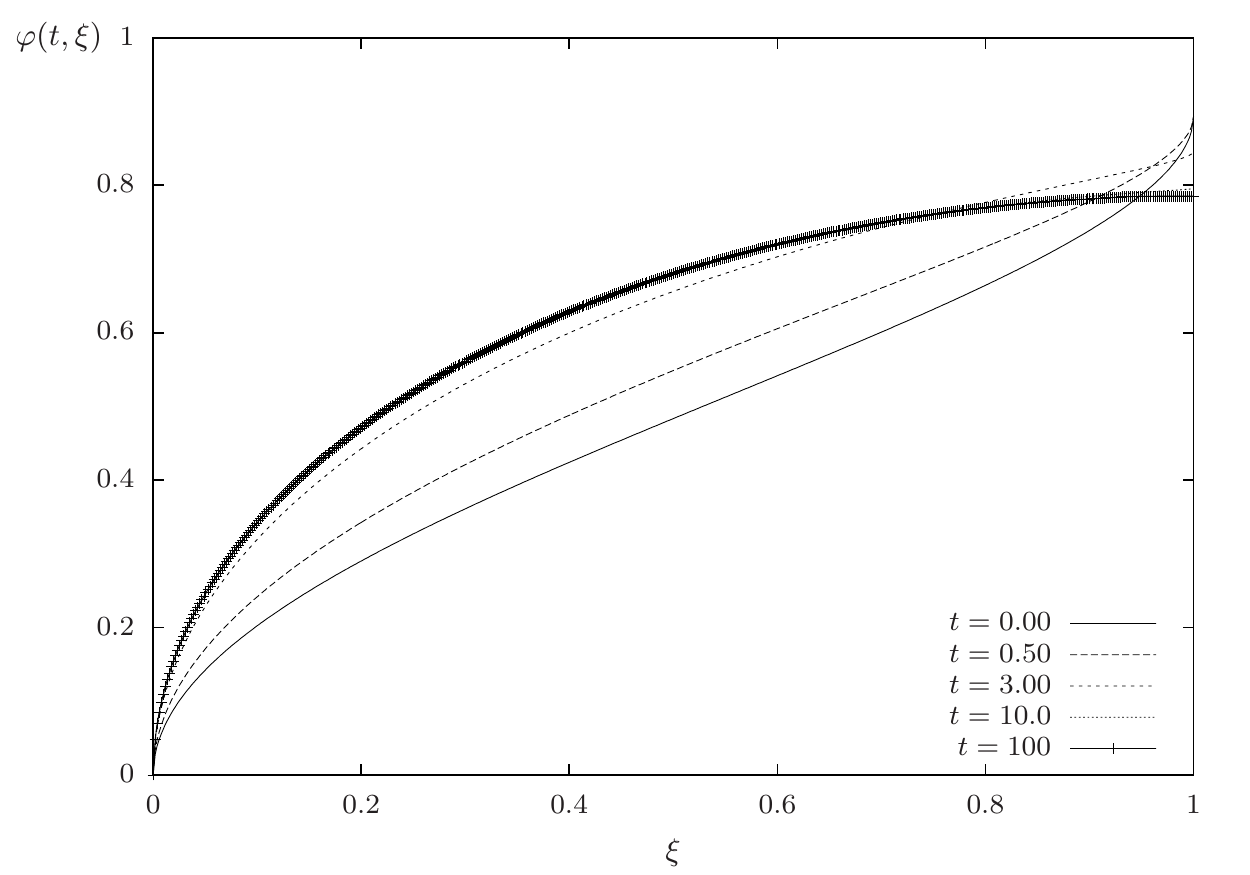}
\includegraphics[scale=0.55]{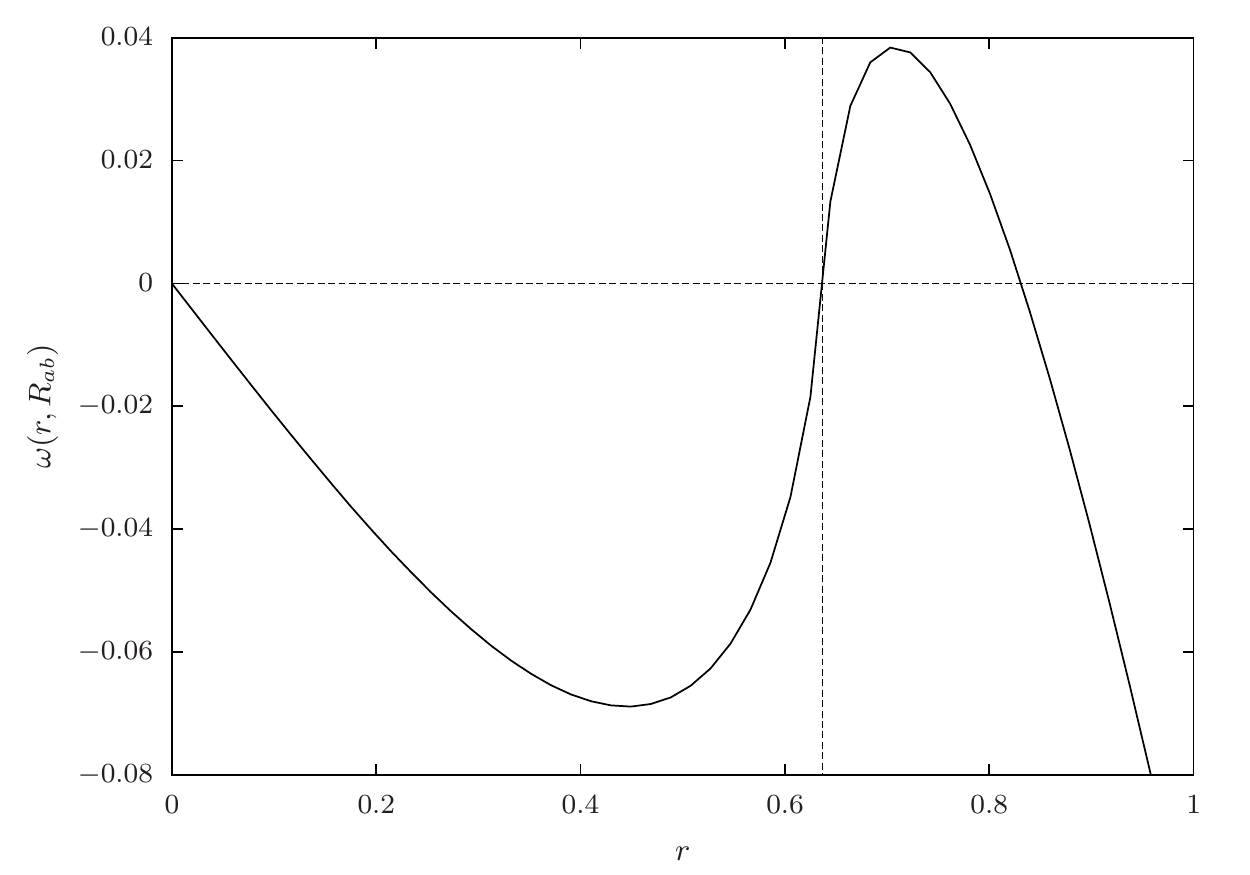}
\caption{Case $W(x)=\frac{|x|^2}{2}-|x|$. Left: Evolution of
$\xi\mapsto\varphi(t,\xi)$ towards a stationary profile, possibly
an integrable function. Right: Velocity field
$r\mapsto\omega(r,R_{ab})$ with the vertical line pointing out
$R_{ab}\sim 0.6366$.} \label{fig:singularfigure}
\end{figure}

\subsection*{Instability Case for the Spherical Shell: $W(x)=\frac{|x|^2}{2}-|x|$.}

In this case, the powers are in the instability region of
Figure~\ref{fig:planes}, below the curve $b=\frac{a}{a-1}$. Then,
due to the results in Theorem~\ref{theorem:instabilitypowers}, a
spherical shell is unstable. One can notice on
Figure~\ref{fig:singularfigure} that the function
$r\mapsto\omega(r,R_{ab})$ associated to the potential
$W(x)=\frac{|x|^2}{2}-|x|$ satisfies
$\partial_1\omega(R_{ab},R_{ab})>0$, so that the instability
condition of Theorem~\ref{fat-inst1} is indeed satisfied.

Figure~\ref{fig:singularfigure} shows that the solution seems to
converge to some stationary state which does not have any singular
part, i.e., possibly an integrable function. Numerically, this
behavior appears for any powers $a,\,b$ in the instability region
of Figure~\ref{fig:planes}. We conjecture that in this region
there exists integrable radial stationary states which are locally
stable under radial perturbations. This has already been proved in
the particular case of $b=2-N$ and $a\geq 2$ in \cite{FHK}. Some
numerical simulations using particle systems done in \cite{KSUB}
however suggest that these stationary states might be unstable for
non radial perturbations.

\subsection*{Energy dissipation}
We remind that the energy functional is given by
\begin{equation*}
E[\rho](t)=\iint_{\real^N\times
\real^N}W(x-y)\rho(t,x)\rho(t,y)\,dy\,dx \,.
\end{equation*}
Using the polar change of coordinates $x=r\sigma$ and
$y=s\tilde{\sigma}$ and using the radial symmetry of
$\rho(t,\cdot)$, this energy writes:
\begin{equation*}
E[\hat{\rho}](t)=\frac{1}{2\sigma_N}\iint_{\real_+^2}\int_{\partial B(0,1)}W(r\sigma-se_1)\hat{\rho}(t,r)\hat{\rho}(t,s)d\sigma
\,ds\,dr \,.
\end{equation*}

\begin{figure}[h]
\includegraphics[scale=0.55]{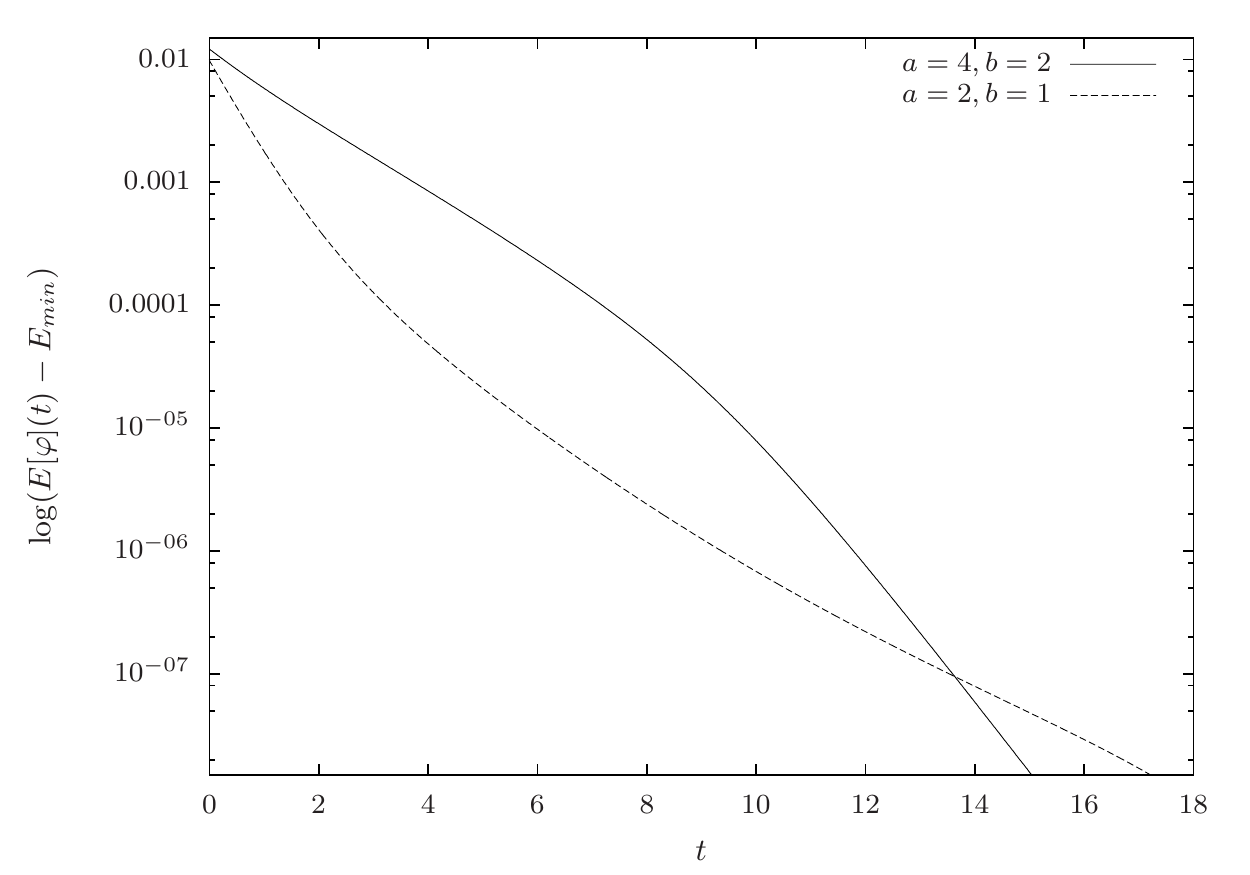}
\caption{Energy decay in logarithmic scale for the regular
repulsive-attractive potential, case $a=4$ and $b=2$ (solid line)
and for the singular repulsive-attractive potential, case $a=2$
and $b=1$ (dashed line). Note that $E_{min}$ is the numerical
limit of the energy as $t\to\infty$.} \label{fig:energy1decay}
\end{figure}

A formal calculation implies that the derivative w.r.t. time of
the energy is negative and given by
\[
\frac{d}{dt}E[\hat{\rho}](t)=-\int_{\real^+}\hat{\rho}(t,r)\hat{v}(t,r)^2
\,dr \, ,
\]
the energy should then decrease in time. Using radially
symmetric coordinates, the energy functional for the inverse
distribution function is given by
\begin{equation}\label{inverseenergy}
E[\varphi](t)=\frac{1}{2\sigma_N}\int_{0}^{1}\int_{0}^{1}\int_{\partial
B(0,1)}W(\varphi(t,\xi)\sigma-\varphi(t,\tilde\xi)e_1)d\sigma
d\tilde\xi d\xi \, .
\end{equation}
We have computed the energy using the formula
\eqref{inverseenergy} to check numerically, in each case, that the
energy decreases. In Figure~\ref{fig:energy1decay} we observe the
exponential decay of the energy for the two numerical examples
presented above for repulsive-attractive potentials.


\section{Appendix}

Let us start by some differential geometry facts. For the sake of
clarity, we first define the type of hypersurfaces we will work
with.

\begin{definition}\label{defmani}
$\mathcal M\subset\real^N$ is a $C^2$ hypersurface (manifold of
dimension $N-1$) if for any $\bar x\in \mathcal M$ there exists a
$C^2$ chart $(U,\varphi)$, i.e., a pair of an open connected set
and a $C^2$ diffeomorphism $\varphi: U \longrightarrow \real^N$,
with $\bar x\in U \subset \real^N$ such that $\varphi(\bar x)=0$
and $y\in \mathcal M\cap U$ if and only if $\varphi(y)\in
\{0\}\times \mathbb R^{N-1}$.
\end{definition}

We will need some technical result from differential geometry in
order to deal with the regularity of the function $\omega$ in
\eqref{omega-def2} and its generalizations to any compact
hypersurface. Note first that if $\mathcal M\subset\real^N$ is a
hyperplane then $M\cap\partial B(x,r)$ is a $N-2$ dimensional
sphere of radius $(r^2-\mbox{dist}(x,\mathcal M)^2)_+^{1/2}$ and
therefore its surface area is
$$
|\mathcal M\cap\partial B(x, r)|_{\mathcal H^{N-2}}=
\sigma_{N-1}(r^2-\mbox{dist}(x,\mathcal M)^2)_+^{\frac{N-2}{2}}
$$
where $\mathcal H^d$ is the d-dimensional Hausdorff measure, and
we remind that $\sigma_{N-1}$ is the surface area of the unit
sphere in $\real^{N-1}$.

The following result is a classical consequence of uniform graphs
lemmas in differential geometry. They state that a compact regular
hypersurface can be covered by graphs with bounds on their
derivatives depending only on the uniform bound of the second
fundamental form. We refer to \cite[Lemma 4.1.1]{PRminimal}. This
allows to show that the volume elements locally converge to those
of a hyperplane in a uniform manner.

\begin{lemma}\label{propmanifolds}
Let $\mathcal M\subset\real^N$ be  a $C^2$ compact hypersurface of
dimension $N-1$ immersed in $\real^N$. Then there exist small
enough $r_0>0$ and constants $C,\tilde C>0$ depending on the
global bound of the second fundamental form of $\mathcal M$ such
that for all $0<r\leq r_0$, and all $x\in \real^N$ with
$\mbox{dist}(x,\mathcal M)<r_0$
\begin{equation}\label{key}
\tilde C \,(r{^2}-\dist(x,\mathcal M)^{2})_+^{\frac{N-2}{2}}
\leq |\mathcal M \cap
\partial B(x, r)|_{\mathcal H^{N-2}} \leq C \,(r{^2}-\dist(x,\mathcal M)^{2})_+^{\frac{N-2}{2}}.
\end{equation}
\end{lemma}

\begin{remark}\label{sphereremark}
Let us note that the previous Lemma is trivial in the case of
$\mathcal M= \partial B(0,\eta)$ for any $\eta >0$ since the
intersection of two $(N-1)$-dimensional spheres of different
radius is always a $(N-2)$-dimensional sphere lying on a
hyperplane. In fact, we can easily compute that if two spheres
$\partial B(0,\eta)$ and $\partial B(x, r)$  intersect, that is
$\abs{|x|-\eta} \le r$, then
$$
|\mathcal \partial B(0,\eta) \cap
\partial B(x, r)|_{\mathcal H^{N-2}}= \sigma_{N-1} r_1^{N-2}
$$
where $r_1=r_1( \eta,r, \dist(x, \partial B(0,\eta)))$ is
the radius of the intersection, which is computable:
\begin{equation*}
r_1=\eta\sqrt{1-\left(\frac
{|x|^2+\eta^2-r^2}{2|x|\eta}\right)^2}\sim
\sqrt{\frac\eta{|x|}}\sqrt{r^2-\dist(x, \partial B(0,\eta))^2},
\end{equation*}
as $r-\dist(x, \partial B(0,\eta))\to 0$. The constants
$r_0$, $C$, and $\tilde C$ of Lemma~{\rm \ref{sphereremark}} can then be
taken uniform for variations of the radius in bounded intervals,
i.e., for $0<\eta_1< \eta<\eta_2$.
\end{remark}

We now can deal with the continuity of the velocity fields
generated by probability densities concentrated on manifolds.
Recall that $\real_+=(0,+\infty)$.

\begin{lemma}\label{reg1}
Let $\mathcal M\subset \mathbb R^N$ be a compact $C^2$
hypersurface, $\mu$ a probability distribution such that $\bar
\mu=\phi \delta_{\mathcal M}$, where $\phi\in L^\infty(\mathcal
M)$, and $g\in C(\real^N/\{0\})$ a radially symmetric function
which is locally integrable on hypersurfaces. Then, the function
$$
\upsilon(x)= \int_{\real^N} g(x-y) \, d\bar \mu(y)
$$
is continuous in $x\in \real^N$. Moreover, the same results hold
while replacing $g(x)$ by a non-radially symmetric function $G\in
C(\real^N/\{0\})$ such that $|G(x)|\leq |g(x)|$, where $g$
satisfies the properties above.
\end{lemma}

\begin{proof}
It is straightforward to check that $\upsilon(x)$ is continuous
for all $\bar x\notin \mathcal M$. Let $\bar x\in \mathcal M$ and let
$r_0$ be given by Lemma \ref{propmanifolds}. For
$0<\varepsilon<r_0$, let $\chi_\varepsilon\in C^\infty(\mathbb R_+)$ be a
cut-off function, such that $\chi_\varepsilon=1$ on
$[0,\varepsilon/2]$, and $\chi_\varepsilon=0$ on
$[\varepsilon,\infty)$. The function $\upsilon$ can then be
written as
\begin{align}
\upsilon(x) &= \int_{\mathcal M} g(x-y)\chi_\varepsilon(x-y)
\, d\bar\mu(y) + \int_{\mathcal M}
g(x-y)[1-\chi_\varepsilon(x-y)] \, d\bar\mu(y) \nonumber\\
&:= \upsilon_1^\varepsilon(x)+\upsilon_2^\varepsilon(x).
\label{cutoff}
\end{align}
It is clear that $\upsilon_2^\varepsilon$ is continuous on $x\in
\real^N$, since $g$ is continuous away from the origin and
$\mbox{\rm supp }(\bar \mu)=\mathcal M$ is compact. Moreover,
given the set $U=\{x\in \real^N : \dist(x,\mathcal
M)<r_0\}$, we can estimate for all $x\in U$
\begin{align*}
|\upsilon_1^\varepsilon(x)|&\le \int_{\real^N} |g(x-y)| \, \chi_\varepsilon(x-y) \, d\bar\mu(y)
\leq \int_{B(x,\varepsilon)} |g(x-y)|  \, d\bar\mu(y)\\
&\leq \|\phi\|_{L^\infty(\mathcal M)} \int_0^\varepsilon
|\hat{g}(r)| \; |\{y\in \mathcal M ; \,|y-x|=r \}|_{\mathcal H^{N-2}} \, dr
\leq C \|\phi\|_{L^\infty(\mathcal M)} \int_0^\varepsilon
|\hat{g}(r)| \;  r^{N-2} \, dr
\end{align*}
where \eqref{key} is used. Moreover, by construction
$v_1^\varepsilon (x)=0$ for all $x\notin U$ for $\varepsilon <
r_0$. Therefore, due to the integrability over hypersurfaces of
$g$, then
$$
\lim_{\varepsilon\to 0}
\|\upsilon_1^\varepsilon\|_{L^\infty(\real^N)} = 0 \, .
$$
This is enough to show the continuity of $\upsilon$ on $\mathcal
M$: for any $\delta>0$, there exists $\varepsilon>0$ such that
$\|\upsilon_1^\varepsilon\|_{L^\infty(\real^N)}\leq \frac \delta
2$. Since $\upsilon_2^\varepsilon$ is continuous, there exists
$\kappa>0$ such that
$|\upsilon_2^\varepsilon(x)-\upsilon_2^\varepsilon(\bar x)|\leq
\frac \delta 2$ if $|x-\bar x|\leq \kappa$. Then,
$|\upsilon(x)-\upsilon(\bar x)|\leq \delta $ if $|x-\bar x|\leq
\kappa$. The last part of the proof is an adaptation of the
previous arguments since the integral inside the norm is less or
equal than $v_1^{2\epsilon}$.
\end{proof}

Now, we want to obtain the continuity with respect to the
hypersurface for the velocity fields associated to measures
concentrated on them. We restrict to the case of spheres since we
only need this particular case. The proof uses the transport
distance $d_\infty$. We remind the reader that it is introduced in
Section 3.

\begin{lemma}\label{reg2}
Let $\mathcal M_\eta:=\partial B(0,\eta)$ and $\bar \mu_\eta
=\phi_\eta \delta_{\mathcal M_\eta}$ be probability measures such
that $\phi_\eta \in L^\infty(\mathcal M_\eta)$ with $0<\eta$. Let
$g\in C^1(\real^N\backslash\{0\})$ be a radially symmetric
function which is locally integrable on hypersurfaces. If the
functions $\phi_\eta$ are uniformly bounded in $\eta$ and
$d_\infty(\bar \mu_{\eta},\bar \mu_{\tilde \eta})\to 0$ as
$\eta-\tilde\eta\to 0$, then
\begin{equation*}
v(x, \eta)= \int_{\partial B(0,\eta)} g(x-y)  \, d\bar \mu_\eta(y)
\end{equation*}
is continuous in  $\real^N \times \real_+$. Moreover, the same
result holds while replacing $g(x)$ by a non-radially symmetric
function $G\in C^1(\real^N\backslash \{0\})$ such that $|G(x)|\leq
|g(x)|$ with the properties above.
\end{lemma}

\begin{proof}
Lemma \ref{reg1} implies directly the continuity with respect to
$x$ for all fixed $\eta$. Using the Remark \ref{sphereremark} and
the proof of Lemma \ref{reg1}, it can be easily checked that this
continuity in $x$ is uniform in $\eta$. Indeed  $|v_1^\epsilon|$
can be made small uniformly in $\eta$ and, due to the estimate $
|\nabla v_2^\epsilon(x)| \le \sup_{\partial B(x,\eta)} \left|
\nabla [g (1-\chi_\epsilon)] \right| $, $v_2^\epsilon$ is
continuous uniformly in $\eta$. Therefore, we only need to show
the continuity in $\eta$ of $v$ for a fixed $x \in \real^N$.

As in the proof of Lemma \ref{reg1}, let $r_0$ be as obtained in
Remark \ref{sphereremark} uniform in $0<\eta_1<\eta<\eta_2$.
We choose again $0<\varepsilon<r_0$ and
$\chi_\epsilon\in C^\infty(\mathbb R_+)$ a cut-off function, such that
$\chi_\varepsilon=1$ on $[0,\varepsilon/2]$, and
$\chi_\varepsilon=0$ on $[\varepsilon,\infty)$. We can write
$\upsilon(x, \eta) = \upsilon_1^\varepsilon(x,
\eta)+\upsilon_2^\varepsilon(x, \eta)$ analogously to
\eqref{cutoff}. As in Lemma \ref{reg1} using the properties of $g$
and the uniformity in Remark \eqref{sphereremark}, we can easily
show that
$$
\lim_{\varepsilon\to 0}
\|\upsilon_1^\varepsilon(\cdot,\eta)\|_{L^\infty(\real^N)} = 0 \,
.
$$
uniformly in $0<\eta_1<\eta<\eta_2$. Therefore, for any
$\delta>0$, there exists $r_0>\varepsilon>0$ such that
$\|\upsilon_1^\varepsilon(\cdot,\eta)\|_{L^\infty(\real^N)}\leq
\frac \delta 4$ uniformly in $0<\eta_1<\eta<\eta_2$. Now, we
estimate
\begin{align*}
|v(x, \eta)-v(x,\tilde \eta)| \leq &\, \|\upsilon_1^\varepsilon(\cdot,\eta)\|_{L^\infty(\real^N)} + \|\upsilon_1^\varepsilon(\cdot,\tilde \eta)\|_{L^\infty(\real^N)} \\
&\, +\left|\int_{\real^N} g(x-y)[1-\chi^\varepsilon(|x-y|)]\,
d(\bar\mu_{\eta}-\bar\mu_{\tilde \eta})(y)\right| \\
\leq &\,  \frac{\delta}2 + \| \nabla
[g(1-\chi^\varepsilon)]\|_{L^\infty(\Delta)}\,
d_\infty(\bar\mu_{\eta},\bar\mu_{\tilde \eta}),
\end{align*}
where $\Delta$ is the convex hull of the set $\{x\}-(\textrm{supp
}\bar \mu_\eta)\cup(\textrm{supp }\bar\mu_{\tilde\eta})$. Notice
that the set $\Delta$ is uniformly bounded in $\eta$ and
$\tilde\eta$.

This estimate shows the continuity in $\eta$ since
$d_\infty(\bar\mu_{\eta},\bar\mu_{\tilde \eta})\to 0$ as $\eta\to
\tilde \eta$, and thus, the last term is bounded by $\delta/2$
provided that $\eta$ is close enough to $\tilde \eta$. Again, the
final part of this Lemma is a small variation of the previous
arguments.
\end{proof}

Finally, we complete the results by showing that if the function
is not locally integrable on hypersurfaces then the velocity field is not
bounded.

\begin{lemma}\label{notreg}
Let $\mathcal M_\eta:=\partial B(0,\eta)$ and $\bar \mu_\eta
=\phi_\eta \delta_{\mathcal M_\eta}$ be probability measures such
that $\phi_\eta(x)\geq \phi_0>0$ for all $\eta_1 \le \eta \le
\eta_2$. Let $g\in C(\real^N\backslash\{0\})$ be a nonnegative
radially symmetric function which is not locally integrable on
hypersurfaces. Then For all $M>0$ there exists $\delta>0$ such
that $$\dist(x,\mathcal M_\eta)<\delta \Longrightarrow
\int_{\real^N} g(x-y) \, d\bar \mu_\eta(y) \ge M \qquad \text{ for
all  $x\in \real^N$ and for all $\eta_1 \le \eta \le \eta_2$.}
$$
\end{lemma}
\begin{proof}
Using Lemma \ref{propmanifolds} and Remark \ref{sphereremark}, for
$x\in \real^N$ with $\dist(x,\mathcal M_\eta)<r_0$, we get
\begin{align*}
\int_{\real^N}  g(x-y) \, d\bar\mu_\eta(y) &\geq \int_{|x-y|<r_0}
g(x-y)  \, d\bar\mu_\eta(y) \geq \phi_0\int_0^{r_0}
\hat{g}(r)\,|\{y\in \mathcal M_\eta; \,|y-x|=r \}|_{\mathcal
H^{N-2}} \, dr \\ & \geq \phi_0\tilde C\int_0^{r_0} \hat{g}(r) \,
(r{^2}-\dist(x,\mathcal M_\eta)^{2})_+^{\frac{N-2}{2}} \, dr\, .
\end{align*}
Since $\int_0^1 \hat{g}(r)r^{N-2}\,dr=+ \infty$ and $g$ is
continuous and nonnegative on $(0,1]$, we deduce that
$$
\lim_{\dist(x,\mathcal M_\eta) \to 0} \int_0^{r_0}
\hat{g}(r) \, (r{^2}-\dist(x,\mathcal
M_\eta)^{2})_+^{\frac{N-2}{2}} \, dr = +\infty\, ,
$$
by the monotone convergence theorem, which conclude the proof.
\end{proof}

\subsection*{Acknowledgments}
DB and JAC were supported by the projects  Ministerio de Ciencia e
Innovaci\'on MTM2011-27739-C04-02 and 2009-SGR-345 from Ag\`encia
de Gesti\'o d'Ajuts Universitaris i de Recerca-Generalitat de
Catalunya. GR was supported by Award No. KUK-I1- 007-43 of Peter
A. Markowich, made by King Abdullah University of Science and
Technology (KAUST). DB, JAC and GR acknowledge partial support
from CBDif-Fr ANR-08-BLAN-0333-01 project. TL acknowledges the
support from NSF Grant DMS-1109805. The authors warmly thank
Joaqu{\'\i}n P\'erez in helping them with the differential
geometry question related to Lemma \ref{propmanifolds}.

\bibliographystyle{plain}
\bibliography{refs}

\end{document}